\documentclass[11pt]{article}

\usepackage{amsmath,amssymb,amsthm,mathrsfs}
\usepackage{booktabs}
\usepackage{enumitem}
\usepackage{tikz}
\usetikzlibrary{automata,positioning,arrows.meta,calc}
\usepackage[margin=1in]{geometry}
\usepackage{hyperref}

\theoremstyle{plain}
\newtheorem{theorem}{Theorem}[section]
\newtheorem{proposition}[theorem]{Proposition}
\newtheorem{lemma}[theorem]{Lemma}
\newtheorem{corollary}[theorem]{Corollary}

\theoremstyle{definition}
\newtheorem{remark}[theorem]{Remark}
\newtheorem{example}[theorem]{Example}

\theoremstyle{definition}
\newtheorem{definition}[theorem]{Definition}
\newtheorem{question}[theorem]{Open Question}

\newcommand{\fd}{\mathfrak{d}}
\newcommand{\fS}{\mathfrak{S}}
\newcommand{\cM}{\mathcal{M}}

\newcommand{\kl}{r}

\begin{document}

\title{Biorthogonal eigenvectors of the Holte carry matrix and cascade-free enumeration}

\author{Daniel Andreas Moj\\
\small Independent Researcher\\
\small Alte Mainzer Str.\ 151, 55118 Mainz, Germany\\
\small \texttt{daniel.moj.apo@gmail.com}}
\date{}

\begin{abstract}
For $k$-summand base-$N$ addition, the carry process is a Markov chain
on $\{0,\ldots,k{-}1\}$ whose transition matrix---the Holte matrix $T$---has
eigenvalues $\{N^{-j}\}_{j=0}^{k-1}$, all simple and independent of~$N$.
We give the complete biorthogonal eigenvector system.
The left eigenvectors factor as
$\sum_i u_j[i]\,x^i = c_{k,j}\,(x{-}1)^j A_{k-j}(x)$,
where $c_{k,j} = |s(k,k{-}j)|/k!$ involves unsigned Stirling numbers
and $A_n(x)$ is the Eulerian polynomial.
The right eigenvectors satisfy
$\sum_i \binom{k{-}1}{i}\,v_j[i]\,x^i = (1{+}x)^{k-1-j}\,Q_j(x)$,
where the quotient polynomials $Q_j$ have palindrome symmetry
$x^j Q_j(1/x) = (-1)^j Q_j(x)$ and converge to $(1{-}x)^j$
as $k \to \infty$; for $j \le 3$, we give explicit closed forms
in terms of~$k$.

The cascade-free avoidance count satisfies
$a(L) = (\sqrt{\fd})^L U_L(x)$
(Chebyshev polynomial of the second kind) whenever the restricted
transfer matrix has dimension $d \le 2$; we prove this is sharp:
for $k$-summand addition, Chebyshev form holds for $k = 3$ and fails
for $k \ge 4$.
The proof uses oscillatory matrix theory to establish non-vanishing
of all spectral residues.
The characteristic polynomial of the restricted transfer matrix is
determined in closed form by a Stirling-weighted Lagrange interpolation
at the Holte eigenvalues.

Two systems with binary carry state spaces are shadow-equivalent
if and only if they share the pair $(N, \fd)$.
The general classification for $k$-state systems reduces to
the characteristic polynomial of $T$.
\end{abstract}



\maketitle

\section{Introduction}
\label{sec:intro}

When $k$ integers are added in base~$N$, digit-wise addition modulo~$N$
produces a carry sequence that propagates from one digit position to the next.
This carry sequence forms a Markov chain on $\{0,\ldots,k-1\}$, and its
transition matrix---the \emph{Holte matrix}~\cite{Holte1997}---has spectral
properties that encode the combinatorics of carry propagation.

Holte~\cite{Holte1997} showed that the eigenvalues of $T$ are
$\{N^{-j}\}_{j=0}^{k-1}$, all simple, and that the eigenvectors do not
depend on the base~$N$.  Diaconis and Fulman~\cite{DiaconisFulmanJACO}
identified the right eigenvectors with the Eulerian idempotents of the
descent algebra of~$\fS_k$ and the left eigenvectors with the Foulkes
character table of~$\fS_k$ (see also Foulkes~\cite{Foulkes1980} and
Novelli--Thibon~\cite{NovelliThibon2012}), connecting carries to riffle
shuffles and symmetric functions~\cite{DiaconisFulman2009, DiaconisFulman2012}.
The arithmetic of positional representations and carry propagation is
treated classically in~\cite{Knuth1997}.

The central theme of this paper is the algebraic-combinatorial
structure of the eigenvectors of~$T$ and their connection to the
representation theory of the symmetric group~$\fS_k$.
Diaconis and Fulman~\cite{DiaconisFulmanJACO} identified the right
eigenvectors of~$T$ with the Eulerian idempotents of the descent
algebra of~$\fS_k$ and the left eigenvectors with the Foulkes
character table.  We make this identification explicit and
computationally effective by giving the complete biorthogonal
eigenvector system in closed form.
The left eigenvectors admit a generating-function factorization into
Stirling numbers of the first kind and Eulerian polynomials
(Proposition~\ref{prop:stirling-eulerian}); the normalization constants
$c_{k,j} = |s(k,k{-}j)|/k!$ are the diagonal entries of the Foulkes
character table and are determined analytically from the spectral
expansion of the carry return probability.
The right eigenvectors satisfy a binomial-palindromic
characterization involving quotient polynomials $Q_j$ that
describe the projection of the Eulerian idempotents onto
the carry state space
(Proposition~\ref{prop:right-ev},
Proposition~\ref{prop:Qj-closed}).
The entry-wise left eigenvector formula is equivalent to results of
Holte~\cite{Holte1997} and Foulkes~\cite{Foulkes1980}
(see also~\cite{DiaconisFulmanJACO}); the generating-function packaging
and the right eigenvector characterization via~$Q_j$ are new.

Two applications of this spectral theory are developed.
First, we prove a Chebyshev threshold dichotomy for cascade-free
avoidance counts (Theorem~\ref{thm:cheb-threshold}): the count
$a(L)$ admits a Chebyshev-$U_L$ parametrization if and only if the
restricted transfer matrix has dimension $d \le 2$.  For $k$-summand
addition, this means Chebyshev form holds for $k = 3$ and fails
for $k \ge 4$.
The characteristic polynomial of the restricted transfer matrix
is determined in closed form by a Stirling-weighted Lagrange
interpolation at the Holte eigenvalues
(Proposition~\ref{prop:stirling-lagrange}), where the interpolation
weights are the Foulkes character values~$c_{k,j}$; this formula
expresses a spectral invariant of the restricted matrix as a
representation-theoretic quantity of~$\fS_k$.

Second, we classify carry chains by the similarity class of their
transfer matrix: two binary-state systems are equivalent if and only
if they share $(N, \fd)$ (Theorem~\ref{thm:stochastic-shadow}),
and the general classification reduces to the characteristic polynomial
(Theorem~\ref{thm:stoch-general}).

Throughout, the spectral analysis of the carry chain
(Sections~\ref{sec:chebyshev}--\ref{sec:classification})
is driven by the representation-theoretic structure
developed in Sections~\ref{sec:spectrum}--\ref{sec:eigenvectors}:
the Foulkes character values~$c_{k,j}$ serve as interpolation weights
in the Stirling--Lagrange formula, and the biorthogonal eigenvector
system controls both the Chebyshev threshold and the classification
invariants.

\medskip
\noindent\textbf{Relationship to the companion paper.}
The companion paper~\cite{Moj2026} develops the transfer-matrix
framework for cascade-free counting in full generality.
Theorem~\ref{thm:universality} is reproved here
(Section~\ref{sec:carry});
Theorems~\ref{thm:chebyshev}
and~\ref{thm:scaling} are proved there and stated here without proof.
The dispersion formula cited in Remark~\ref{rem:fib-poisson} is also
from~\cite{Moj2026}.
The Chebyshev threshold (Theorem~\ref{thm:cheb-threshold}) and all
results in Sections~\ref{sec:spectrum}--\ref{sec:classification}
are proved entirely within this paper.

\section{The Carry Chain}
\label{sec:carry}

\subsection{Transfer matrix for $k$-summand addition}

Consider $k$ independent digits $d_1,\ldots,d_k$ drawn uniformly
from $\{0,\ldots,N-1\}$. For incoming carry $c \in \{0,\ldots,k-1\}$,
the outgoing carry is $c' = \lfloor(d_1 + \cdots + d_k + c)/N\rfloor$.
The entry formula of the \emph{Holte matrix} $T \in \mathbb{R}^{k \times k}$ is
\begin{equation}
\label{eq:holte-entry}
T[c',c] = \frac{1}{N^k}\,\sum_{s=c'N-c}^{(c'+1)N-c-1} B_k(s),
\end{equation}
where $B_k(s) = \#\{(d_1,\ldots,d_k)\in\{0,\ldots,N-1\}^k : \sum d_i = s\}$
is the $k$-fold convolution of the discrete uniform distribution.
The \emph{count matrix} is $T_{\mathrm{count}} := N^k T$.

For the two-summand case ($k = 2$), each digit either
\emph{generates} carry ($d \in \mathrm{Gen}$, $|\mathrm{Gen}| = g$),
\emph{propagates} the current carry state ($d \in \mathrm{Prop}$,
$|\mathrm{Prop}| = t$), or \emph{kills} carry to~$0$
($d \in \mathrm{Kill}$, $|\mathrm{Kill}| = \kl$), with $g + t + \kl = N$.
The $2 \times 2$ transfer matrix is
\begin{equation}
\label{eq:transfer}
T = \begin{pmatrix} \kl + t & g \\ \kl & g \end{pmatrix}, \qquad
\operatorname{tr}(T) = N, \quad \det(T) = tg =: \fd.
\end{equation}

\begin{figure}[ht]
\centering
\begin{tikzpicture}[
    >={Stealth},
    state/.style={circle, draw, thick, minimum size=28pt,
                  inner sep=1pt, font=\small},
    every edge/.style={draw, thick, ->, >=Stealth},
    lbl/.style={font=\small, fill=white, inner sep=1.5pt}
  ]
  \node[state] (s0) at (0,0) {$0$};
  \node[state] (s1) at (4,0) {$1$};
  \node[below=2pt, font=\footnotesize] at (s0.south) {no carry};
  \node[below=2pt, font=\footnotesize] at (s1.south) {carry};
  \path (s0) edge[loop above, looseness=8, in=120, out=60]
        node[lbl, above=4pt] {$\frac{\kl+t}{N}$} (s0);
  \path (s1) edge[loop above, looseness=8, in=120, out=60]
        node[lbl, above=4pt] {$\frac{g+t}{N}$} (s1);
  \path (s0) edge[bend left=20]
        node[lbl, above] {$\frac{g}{N}$} (s1);
  \path (s1) edge[bend left=20]
        node[lbl, below] {$\frac{\kl}{N}$} (s0);
\end{tikzpicture}
\caption{The carry chain for binary state spaces ($k = 2$).
Edge labels are transition probabilities under uniform digits.}
\label{fig:carry-chain}
\end{figure}
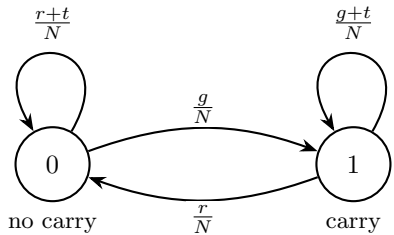

\begin{lemma}[Transfer matrix invariants]
\label{lem:transfer-inv}
The characteristic polynomial $\lambda^2 - N\lambda + \fd = 0$ depends only on $(N, \fd)$.
\end{lemma}

A sequence $x \in \{0,\ldots,N-1\}^L$ is \emph{cascade-free}
if the carry never reaches its maximum value $k - 1$; the count of
such sequences is $a(L) = \mathbf{1}^\top \widetilde{T}^L e_0$, where
$\widetilde{T} := T|_{\{0,\ldots,k-2\}}$ is the \emph{restricted
transfer matrix} (rows and columns indexed by the non-forbidden states)
and $e_0$ is the indicator of state~$0$.

\subsection{Cascade-free sequences}

The transfer matrix method for counting constrained sequences goes back
to~\cite{Stanley1997}; the cascade-free sequence results below
are developed in the companion paper~\cite{Moj2026}.

\begin{remark}[Background from the companion paper]
\label{rem:reprinted}
Theorem~\ref{thm:universality} is proved here for completeness;
Theorems~\ref{thm:chebyshev}
and~\ref{thm:scaling}
are proved in~\cite{Moj2026} and stated here without proof.
The dispersion formula cited in Remark~\ref{rem:fib-poisson}
is also from~\cite{Moj2026}.
Theorem~\ref{thm:cheb-threshold} is proved entirely within this paper.
Remark~\ref{rem:fib-poisson},
and Lemmas~\ref{lem:simple-evals}--\ref{lem:nonzero-res} are new.
\end{remark}

\begin{theorem}[Universality~{\cite[Theorems~3.4,~4.2]{Moj2026}}]
\label{thm:universality}
The cascade-free count $a(L) := |\{x \in X^L : x \text{ cascade-free}\}|$ satisfies
$a(L) = N \cdot a(L-1) - \fd \cdot a(L-2)$ with $a(0) = 1$, $a(1) = N$, and depends
only on the pair $(N, \fd)$.
\end{theorem}

\begin{proof}
The restricted transfer matrix $\widetilde{T} \in \mathbb{R}^{2\times 2}$
satisfies $\operatorname{tr}(\widetilde{T}) = N$ and
$\det(\widetilde{T}) = \fd$ (Lemma~\ref{lem:transfer-inv}).
By Cayley--Hamilton, $\widetilde{T}^2 = N\widetilde{T} - \fd\, I$, so
$a(L) = \mathbf{1}^\top \widetilde{T}^L e_0$ satisfies
$a(L) = N\,a(L{-}1) - \fd\,a(L{-}2)$ for $L \ge 2$.
The initial values are $a(0) = \mathbf{1}^\top e_0 = 1$ and
$a(1) = \mathbf{1}^\top \widetilde{T}\, e_0 = N$
(since $\widetilde{T}$ is column-stochastic with column sums~$N$
after multiplying by~$N^k$, and the binary case gives
$(\kl+t) + \kl = N$ directly from $g + t + \kl = N$).
\end{proof}

Define the \emph{coupling parameter} $x := N/(2\sqrt{\fd}) \ge 1$.

\begin{theorem}[Chebyshev representation~{\cite[Theorem~4.2]{Moj2026}}]
\label{thm:chebyshev}
For $\fd > 0$: $a(L) = (\sqrt{\fd})^L \cdot U_L(x)$ where $U_L$ is the Chebyshev polynomial
of the second kind~\cite{MasonHandscomb2003} at $x = N/(2\sqrt{\fd})$.
\end{theorem}

\begin{theorem}[Scaling law and Fibonacci bisection~{\cite[Theorems~5.2,~5.4]{Moj2026}}]
\label{thm:scaling}
\label{thm:fibonacci}
For all odd primes~$p$: $a_{\mathrm{add}}(L) = p^L \cdot a_{\mathrm{dbl}}(L)$.
For base-$3$ doubling, $\fd = 1$, $x = 3/2$, and
$a_{\mathrm{dbl}}(L) = U_L(3/2) = F(2L + 2)$,
the Fibonacci bisection
(OEIS~A001906, \cite{Sloane2025,Koshy2001}).
\end{theorem}

\section{The Holte Spectrum}
\label{sec:spectrum}

\begin{lemma}[Uniformity of $S \bmod N$]
\label{lem:uniform}
For any $k \ge 1$ and any integer $c$, $(S+c) \bmod N \sim \mathrm{Uniform}\{0,\ldots,N-1\}$.
\end{lemma}

\begin{proof}
Let $\omega = e^{2\pi i/N}$. The discrete Fourier coefficient of $\Pr(d_i \equiv r \pmod N)$
at frequency $m \neq 0$ is $\frac{1}{N}\sum_{d=0}^{N-1}\omega^{md} = 0$, since the sum is
a full geometric series. By independence, the same vanishing holds for $S \bmod N$.
The inverse DFT then recovers the uniform distribution. Adding the constant $c$ merely
shifts all residues, leaving the uniform distribution invariant.
\end{proof}

\begin{theorem}[Holte matrix spectrum]
\label{thm:holte-spectrum}
For $k$-summand base-$N$ addition, the Holte transfer matrix $T \in \mathbb{R}^{k \times k}$
satisfies:
\begin{enumerate}[label=(\roman*)]
\item The eigenvalues are $\{N^{-j} : j = 0,\ldots,k-1\}$, all simple.
\item The stationary distribution is $\pi_i = A(k,i)/k!$, where $A(k,i)$ are the Eulerian
numbers.
\item The eigenvectors are independent of~$N$.
\end{enumerate}
\end{theorem}

\begin{proof}
\noindent\textit{Setup.}
Let $d_1,\ldots,d_k \sim \mathrm{Uniform}\{0,\ldots,N-1\}$ be independent, $S = \sum_{i=1}^k d_i$,
and $C' = \lfloor(S+c)/N\rfloor$ for incoming carry $c \in \{0,\ldots,k-1\}$.

\noindent\textit{Part~(ii)---stationary distribution.}
We prove $T\pi = \pi$ for $\pi_c = A(k,c)/k!$ by establishing, for each
$c' \in \{0,\ldots,k-1\}$, the identity
\begin{equation}
\label{eq:stationarity-id}
\sum_{c=0}^{k-1} A(k,c) \sum_{s=c'N-c}^{(c'+1)N-c-1} B_k(s) = N^k A(k,c').
\end{equation}
Dividing by $k!\, N^k$ this is $\sum_c \pi_c T[c',c] = \pi_{c'}$, i.e., stationarity.

The proof uses two classical identities.

\medskip
\noindent\textit{Worpitzky identity~\cite[Prop.~1.4.4]{Stanley1997}}: For any $n \in \mathbb{Z}_{\ge 0}$,
\begin{equation}
\label{eq:worpitzky}
n^k = \sum_{c=0}^{k-1} A(k,c)\binom{n+k-1-c}{k}.
\end{equation}

\noindent\textit{Eulerian number formula~\cite[§6.5]{Comtet1974}}: For each $c' \in \{0,\ldots,k-1\}$,
\begin{equation}
\label{eq:euler-explicit}
A(k,c') = \sum_{j=0}^{c'}(-1)^j\binom{k+1}{j}(c'+1-j)^k.
\end{equation}

\noindent\textit{Step~1: Sum over an interval of $B_k$.}
By inclusion-exclusion, $B_k(s) = \sum_{j \ge 0}(-1)^j\binom{k}{j}\binom{s-jN+k-1}{k-1}$
(terms with $s-jN < 0$ vanish). Applying the hockey-stick identity
$\sum_{s=a}^{a+N-1}\binom{s+m}{m} = \binom{a+N+m}{m+1} - \binom{a+m}{m+1}$ with $m = k-1$
and $a = c'N - c - jN$:
\[
  \sum_{s=c'N-c}^{(c'+1)N-c-1} B_k(s)
  = \sum_{j=0}^{c'}(-1)^j\binom{k}{j}\!
    \left[\binom{(c'{-}j{+}1)N{-}c{+}k{-}1}{k} - \binom{(c'{-}j)N{-}c{+}k{-}1}{k}\right],
\]
where the upper limit $j \le c'$ follows since $\binom{(c'-j)N-c+k-1}{k} = 0$
for $(c'-j)N - c + k - 1 < k$, i.e.\ $j > c'$.

\noindent\textit{Step~2: Sum over $c$ with weights $A(k,c)$.}
Apply~\eqref{eq:worpitzky} with $n = mN$ (for $m \in \mathbb{Z}_{\ge 0}$):
$\sum_c A(k,c)\binom{mN-c+k-1}{k} = (mN)^k$.
Swapping the order of summation and applying this to each bracket in Step~1:
\begin{align*}
\sum_{c=0}^{k-1} A(k,c) \sum_{s=c'N-c}^{(c'+1)N-c-1}B_k(s)
&= \sum_{j=0}^{c'}(-1)^j\binom{k}{j}\!
   \bigl[((c'-j+1)N)^k - ((c'-j)N)^k\bigr] \\
&= N^k\sum_{j=0}^{c'}(-1)^j\binom{k}{j}
   \bigl[(c'-j+1)^k - (c'-j)^k\bigr].
\end{align*}

\noindent\textit{Step~3: Reduction to the Eulerian formula.}
Set $i = c'-j$ (so $j = c'-i$). The right side of Step~2 becomes
\[
N^k\sum_{i=0}^{c'}(-1)^{c'-i}\binom{k}{c'-i}\bigl[(i+1)^k - i^k\bigr].
\]
We show this equals $N^k A(k,c')$.  Splitting $(i+1)^k - i^k$ and using the
Eulerian formula~\eqref{eq:euler-explicit}, it suffices to verify that
\begin{equation}
\label{eq:step3-target}
\sum_{i=0}^{c'}(-1)^{c'-i}\binom{k}{c'-i}\bigl[(i+1)^k - i^k\bigr]
= \sum_{i=0}^{c'}(-1)^{c'-i}\binom{k+1}{c'-i}(i+1)^k.
\end{equation}
To see this, expand the right side using
$\binom{k+1}{c'-i} = \binom{k}{c'-i} + \binom{k}{c'-i-1}$:
\[
\text{RHS of~\eqref{eq:step3-target}}
= \underbrace{\sum_{i=0}^{c'}(-1)^{c'-i}\binom{k}{c'-i}(i+1)^k}_{=:\,\Sigma_1}
+ \underbrace{\sum_{i=0}^{c'}(-1)^{c'-i}\binom{k}{c'-i-1}(i+1)^k}_{=:\,\Sigma_2}.
\]
Re-index $\Sigma_2$ by $m = i+1$ (so $i = m-1$ and $c'-i-1 = c'-m$):
\[
  \Sigma_2 = \sum_{m=1}^{c'+1}(-1)^{c'-m+1}\binom{k}{c'-m}\,m^k
  = -\sum_{m=0}^{c'}(-1)^{c'-m}\binom{k}{c'-m}\,m^k,
\]
where the $m = 0$ term vanishes (since $0^k = 0$ for $k \ge 1$) and the
$m = c'+1$ term vanishes (since $\binom{k}{-1} = 0$).  Combining:
\[
  \Sigma_1 + \Sigma_2
  = \sum_{i=0}^{c'}(-1)^{c'-i}\binom{k}{c'-i}\bigl[(i+1)^k - i^k\bigr],
\]
which is the left side of~\eqref{eq:step3-target}.
This proves~\eqref{eq:stationarity-id} and hence $T\pi = \pi$.

\textit{Part~(i)---eigenvalues.}
We work with $T$ as a linear operator on $\mathcal{F} = \{f:\{0,\ldots,k-1\}\to\mathbb{R}\}$,
writing $(Tf)(c) = E[f(C')\mid C=c]$.  The monomial basis $\{1, c, c^2, \ldots, c^{k-1}\}$
spans $\mathcal{F}$ (the evaluation Vandermonde matrix is invertible). We prove:

\smallskip
\noindent\textit{Claim}: For each $j\in\{0,\ldots,k-1\}$,
$(Tc^j)(c) = N^{-j}c^j + p_{j-1}(c)$
where $p_{j-1}$ is a polynomial of degree $\le j-1$.
\smallskip

In other words, $T$ acts on the polynomial filtration
$\mathcal{P}_0\subset\mathcal{P}_1\subset\cdots\subset\mathcal{P}_{k-1} = \mathcal{F}$
(where $\mathcal{P}_j = \mathrm{span}\{1,c,\ldots,c^j\}$) with each $\mathcal{P}_j$
invariant and with leading eigenvalue $N^{-j}$ on $\mathcal{P}_j/\mathcal{P}_{j-1}$.
A lower-triangular matrix with distinct diagonal entries is diagonalizable, giving
eigenvalues $\{N^{-j}\}_{j=0}^{k-1}$, all simple (distinct since $N\ge 2$).

\smallskip
\noindent\textit{Proof of Claim for $j=1$}: Write $R := (S+c)\bmod N$.
By Lemma~\ref{lem:uniform}, $R \sim \mathrm{Uniform}\{0,\ldots,N-1\}$ with $E[R] = (N-1)/2$,
independently of $c$.  Since $E[S] = k(N-1)/2$:
\begin{align}
(Tc)(c) &= E[C'\mid C=c] = \frac{E[S+c] - E[R]}{N}
          = \frac{k(N-1)/2 + c - (N-1)/2}{N} \notag \\
         &= \frac{c}{N} + \frac{(k-1)(N-1)}{2N}.
\label{eq:first-moment}
\end{align}
The leading coefficient in $c$ is $N^{-1}$.  \checkmark

\smallskip
\noindent\textit{Proof of Claim for general $j$}: We use the decomposition
$C' = (S+c-R)/N$ to write
\[
(Tc^j)(c) = E\!\left[\!\left(\frac{S+c-R}{N}\right)^{\!j}\right]
= N^{-j}\sum_{m=0}^{j}\binom{j}{m}c^{j-m}E\!\left[(S-R)^m\right].
\]
The term $m=0$ contributes $N^{-j}c^j$.  For $m\ge 1$: note that
$S - R = (S+c-R) - c = NC' - c$, so
\[
E[(S-R)^m] = E[(NC'-c)^m]
= \sum_{\ell=0}^{m}\tbinom{m}{\ell}N^{\ell}(-c)^{m-\ell}E[C'^\ell].
\]
Each $E[C'^\ell] = (Tc^\ell)(c)$ is itself a polynomial in $c$ of degree $\le\ell$ by the
induction hypothesis (for $\ell < j$).  The term $\ell = m = j$ is self-referential:
it involves $(Tc^j)(c)$, the quantity being computed.  However, the induction targets
only the \emph{leading coefficient} of $(Tc^j)(c)$, and the $\ell = m = j$ term does not
contribute to it: the coefficient of $c^m$ in $E[(NC'-c)^m]$ is
$\sum_{\ell=0}^{m}\binom{m}{\ell}(-1)^{m-\ell}N^{\ell}\cdot N^{-\ell} = (1-1)^m = 0$
for $m \ge 1$, so no $m \ge 1$ term contributes a $c^j$-monomial.
The induction is therefore closed: the leading term $N^{-j}c^j$ is determined by
the $m = 0$ contribution alone, and the lower-order terms involve only
$(Tc^\ell)(c)$ for $\ell < j$, which are known by hypothesis.
Therefore $E[(S-R)^m]$ is a polynomial in $c$
of degree $\le m$, and the contribution of the $m$-th term to $(Tc^j)(c)$ has degree
$\le (j-m)+m = j$, and the $j$-th degree contribution arises solely from $m=0$.
Hence $(Tc^j)(c) = N^{-j}c^j + (\text{degree} \le j-1)$, completing the induction.

\textit{Part~(iii)---$N$-independence of eigenvectors.}
By the Claim, $T$ acts on the polynomial filtration
$\mathcal{P}_0 \subset \mathcal{P}_1 \subset \cdots \subset \mathcal{P}_{k-1}$
with each $\mathcal{P}_j$ invariant, and the induced action on
$\mathcal{P}_j / \mathcal{P}_{j-1}$ is scalar multiplication by~$N^{-j}$.
Since the eigenvalues $N^{-j}$ are pairwise distinct (for $N \ge 2$), the
eigenspace for $N^{-j}$ is one-dimensional and lies in~$\mathcal{P}_j$.
Let $v_j \in \mathcal{P}_j$ be the unique monic polynomial of degree~$j$
with $Tv_j = N^{-j} v_j$.  Writing $v_j = c^j + \sum_{m < j} a_m c^m$,
the eigenequation determines each $a_m$ by
\[
  (N^{-m} - N^{-j})\, a_m
  = -\bigl[\text{coefficient of $c^m$ in }
    (T - N^{-j}\mathrm{Id})\bigl(c^j + \textstyle\sum_{\ell > m} a_\ell\, c^\ell\bigr)\bigr],
\]
a triangular recursion descending from $m = j{-}1$ to $m = 0$.
The right-hand side involves $a_{m+1}, \ldots, a_{j-1}$ (known by induction)
and the off-diagonal entries of $T$ in the monomial basis---that is, the
lower-degree parts of $(Tc^n)(c)$ for $n \le j$.
Diaconis and Fulman~\cite{DiaconisFulmanJACO} identify the right
eigenvectors of $T$ with the Eulerian idempotents of the descent algebra
of~$\fS_k$, and the left eigenvectors with the Foulkes character table
of~$\fS_k$, both of which are defined purely in terms of the symmetric
group~$\fS_k$ and do not involve~$N$.
(Holte~\cite[Theorem~3]{Holte1997} gave the original proof of
$N$-independence by direct computation of the eigenvectors.)
In particular, the biorthogonal system $(u_j, v_j)$ of
Propositions~\ref{prop:stirling-eulerian} and~\ref{prop:right-ev} is
$N$-independent.
\end{proof}

\begin{remark}[Non-reversibility for $k \ge 4$]
\label{rem:non-reversibility}
For $k = 2$ and $k = 3$ the carry chain satisfies detailed balance
$\pi_c T[c',c] = \pi_{c'} T[c,c']$, so $T$ is self-adjoint in $L^2(\pi)$.
For $k \ge 4$ this fails: direct computation for
$k = 4$, $N = 2$ gives $\pi_0 T[1,0] = 10/384 \ne 11/384 = \pi_1 T[0,1]$.
The carry chain is not reversible for $k \ge 4$.
The proofs of Theorem~\ref{thm:holte-spectrum}(i) and~(iii)
do not require self-adjointness.
\end{remark}

\begin{table}[ht]
\centering
\caption{Holte matrix spectral data (illustrative).}
\label{tab:holte-spec}
\small
\begin{tabular}{@{}cllc@{}}
\toprule
$k$ & $\pi \cdot k!$ & Eulerian numbers $A(k,\,\cdot)$ & Eigenvalues \\
\midrule
$3$ & $[1,4,1]$        & $[1,4,1]$          & $1,\, N^{-1},\, N^{-2}$ \\
$4$ & $[1,11,11,1]$    & $[1,11,11,1]$      & $1,\, N^{-1},\, N^{-2},\, N^{-3}$ \\
$5$ & $[1,26,66,26,1]$ & $[1,26,66,26,1]$   & $1,\, N^{-1},\ldots,\, N^{-4}$ \\
\bottomrule
\end{tabular}
\end{table}

\begin{corollary}[Eigenvalue multiset as complete invariant]
\label{cor:eigenvalue-complete}
The multiset $\{N^{-j}\}_{j=0}^{k-1}$ determines the conjugacy class of $T$ uniquely.
\end{corollary}

\begin{proof}
The ratio $\lambda_0/\lambda_1 = N$ and the cardinality $k$ are read off directly.
Since all eigenvalues are simple and the eigenvectors are $N$-independent
(Theorem~\ref{thm:holte-spectrum}(iii)), the multiset fixes the conjugacy class.
\end{proof}

\section{The Biorthogonal Eigenvector System}
\label{sec:eigenvectors}

Theorem~\ref{thm:holte-spectrum}(iii) establishes $N$-independence of the eigenvectors
but does not give their explicit form. The following two propositions supply the
complete biorthogonal eigenvector system.
Holte~\cite{Holte1997} gave entry-wise formulae for both left and right
eigenvectors; Diaconis and Fulman~\cite{DiaconisFulmanJACO} identified the left
eigenvectors with the Foulkes character table of~$\fS_k$ and the right
eigenvectors with Loday's Eulerian idempotents, giving new closed
forms and computing carry correlations
(see also Novelli--Thibon~\cite{NovelliThibon2012}
for a treatment via noncommutative symmetric functions).
The entry-wise left eigenvector formula is equivalent
to a result of Foulkes~\cite{Foulkes1980}.
The generating-function packaging as a product of Stirling numbers,
binomial shifts, and Eulerian polynomials
(Proposition~\ref{prop:stirling-eulerian}) follows from these known
entry-wise formulae by standard generating-function algebra;
the self-contained proof below establishes the eigenvector equation
via Worpitzky's identity~\eqref{eq:worpitzky} and determines the
Stirling normalization constants from the spectral expansion of the
carry return probability.
The binomial-palindromic characterization of the right eigenvectors,
including the structure of the quotient polynomials~$Q_j$ and the
closed forms for $Q_2$ and~$Q_3$
(Proposition~\ref{prop:right-ev}, Proposition~\ref{prop:Qj-closed}),
appears not to have been observed previously.

\subsection{Left eigenvectors: Stirling--Eulerian factorization}

\begin{proposition}[Stirling--Eulerian factorization of left eigenvectors]
\label{prop:stirling-eulerian}
For $k$-summand base-$N$ addition, the left eigenvector $u_j$
($j = 0,\ldots,k{-}1$) of the Holte matrix $T$ associated with eigenvalue
$N^{-j}$ has generating function
\begin{equation}
\label{eq:left-ev-gf}
\sum_{i=0}^{k-1} u_j[i]\, x^i \;=\; \frac{(-1)^j\,|s(k,\,k{-}j)|}{k!}
\;\cdot\;(x-1)^j\;\cdot\; A_{k-j}(x),
\end{equation}
where $|s(k,m)|$ is the unsigned Stirling number of the first kind and
$A_n(x) = \sum_{i=0}^{n-1} A(n,i)\,x^i$ is the $n$-th Eulerian polynomial.
Entry-wise:
\[
u_j[i] \;=\; \frac{|s(k,\,k{-}j)|}{k!} \;\sum_{m=0}^{j}
\binom{j}{m}\,(-1)^m\, A(k{-}j,\; i{-}m),
\]
where $A(n, i) = 0$ for $i < 0$ or $i \geq n$.
In particular: $u_0[i] = A(k,i)/k!$ is the stationary distribution,
$u_1 = -\binom{k}{2}/k! \cdot [\Delta A_{k-1}]$
(the first finite difference of the $(k{-}1)$-st Eulerian numbers), and
$u_{k-1}[i] = (-1)^i\binom{k{-}1}{i}/k$.
\end{proposition}

\begin{proof}
We verify~\eqref{eq:left-ev-gf} by checking $u_j T = N^{-j}\,u_j$.
Write $\nu = k{-}j$ and $w_j[c] = \sum_{m=0}^{j}\binom{j}{m}(-1)^m A(\nu,\,c{-}m)$
(the $j$-th backward difference of $A(\nu,\cdot)$), so that
$u_j = c_{k,j}\,w_j$ with $c_{k,j} = |s(k,k{-}j)|/k!$.
The eigenvalue equation $u_j T = N^{-j}\,u_j$ is equivalent to
$w_j T = N^{-j}\,w_j$ (the scalar $c_{k,j}$ cancels). We prove
\begin{equation}\label{eq:gen-stationarity}
  \sum_{c=0}^{k-1} w_j[c]\!\sum_{s=c'N-c}^{(c'+1)N-c-1}\!B_k(s)
  = N^{k-j}\,w_j[c']
  \qquad\text{for each } c' \in \{0,\ldots,k{-}1\}.
\end{equation}

\noindent\textit{Step~1} (unchanged from Part~(ii) of Theorem~\ref{thm:holte-spectrum}):
\[
  \sum_{s=c'N-c}^{(c'+1)N-c-1}\!B_k(s)
  = \sum_{\ell=0}^{c'}(-1)^\ell\binom{k}{\ell}\!
    \Bigl[\binom{(c'{-}\ell{+}1)N{-}c{+}k{-}1}{k}
         - \binom{(c'{-}\ell)N{-}c{+}k{-}1}{k}\Bigr].
\]

\noindent\textit{Generalized Step~2.}
We claim: for any integer $q \ge 0$,
\begin{equation}\label{eq:gen-worpitzky}
  \sum_{c=0}^{k-1} w_j[c]\,\binom{qN{-}c{+}k{-}1}{k} = (qN)^\nu.
\end{equation}
Expanding $w_j[c]$ and substituting $c \to c{+}m$ in the inner sum:
\[
\text{LHS} = \sum_{c=0}^{k-j-1} A(\nu,c)\,
\underbrace{\sum_{m=0}^{j}\binom{j}{m}(-1)^m\binom{qN{-}c{-}m{+}k{-}1}{k}}_
{= \binom{qN{-}c{+}\nu{-}1}{\nu}},
\]
where the inner sum equals $\nabla^j_n\binom{n+k{-}1{-}c}{k}\big|_{n=qN}
= \binom{qN{-}c{+}\nu{-}1}{\nu}$ by iterating Pascal's identity.
Then~\eqref{eq:worpitzky} with parameter~$\nu$ at $n = qN$ yields~\eqref{eq:gen-worpitzky}.

Swapping the order of summation in~\eqref{eq:gen-stationarity} and
applying~\eqref{eq:gen-worpitzky} to each bracket:
\[
  \text{LHS of \eqref{eq:gen-stationarity}}
  = N^\nu \sum_{\ell=0}^{c'}(-1)^\ell\binom{k}{\ell}
    \bigl[(c'{-}\ell{+}1)^\nu - (c'{-}\ell)^\nu\bigr]
  =: N^\nu\,G(c').
\]

\noindent\textit{Generalized Step~3.}
We show $G(c') = w_j[c']$.
The Eulerian formula~\eqref{eq:euler-explicit} with parameter~$\nu$ gives
$A(\nu,c') = \sum_\ell (-1)^\ell\binom{\nu+1}{\ell}(c'{+}1{-}\ell)^\nu$.
Applying $\nabla^j$ in $c'$ and using the Vandermonde--Chu identity
$\sum_{m=0}^{j}\binom{j}{m}\binom{\nu+1}{\ell-m} = \binom{k+1}{\ell}$:
\[
  w_j[c'] = \nabla^j A(\nu,c')
  = \sum_{\ell=0}^{c'}(-1)^\ell\binom{k+1}{\ell}(c'{+}1{-}\ell)^\nu.
\]
Decomposing $\binom{k+1}{\ell} = \binom{k}{\ell} + \binom{k}{\ell-1}$
and re-indexing the second sum by $\ell' = \ell{-}1$:
\[
  w_j[c'] = \sum_\ell(-1)^\ell\binom{k}{\ell}(c'{+}1{-}\ell)^\nu
           - \sum_{\ell'}(-1)^{\ell'}\binom{k}{\ell'}(c'{-}\ell')^\nu
  = G(c').
\]
This proves~\eqref{eq:gen-stationarity} and hence $w_j T = N^{-j}\,w_j$.

\medskip
\noindent\textit{Normalization.}
Since all eigenvalues $N^{-j}$ are simple, the left eigenspace for
each~$j$ is one-dimensional, and $w_j$ is determined up to a
scalar multiple.  The off-diagonal biorthogonality
$\langle w_j, v_\ell \rangle = 0$ for $j \ne \ell$
follows from eigenvalue separation:
$N^{-j}\langle w_j, v_\ell\rangle
= \langle w_j T, v_\ell\rangle
= \langle w_j, Tv_\ell\rangle
= N^{-\ell}\langle w_j, v_\ell\rangle$,
so $(N^{-j} - N^{-\ell})\langle w_j, v_\ell\rangle = 0$.
Because $\{v_0,\ldots,v_{k-1}\}$ is a basis of~$\mathbb{R}^k$
and $w_j \ne 0$, the inner product $\langle w_j, v_j \rangle$
cannot vanish (otherwise $w_j$ would be orthogonal to every
basis vector).  Hence there exists a unique scalar~$c_{k,j}$
such that $u_j := c_{k,j}\,w_j$ satisfies
$\langle u_j,\, v_j\rangle = 1$ with
$v_j[0] = 1$ (Proposition~\ref{prop:right-ev}).
We claim $c_{k,j} = |s(k,k{-}j)|/k!$.

\smallskip
\noindent\textit{Step~N1: counting the return probability.}
Let $T_{\mathrm{count}} = N^k T$ be the count matrix.
The $(0,0)$-entry of $T_{\mathrm{count}}^n$ counts $k$-tuples
$(a_1,\ldots,a_k) \in \{0,\ldots,N^n-1\}^k$ for which
the carry chain starts at $c_0 = 0$ and returns to $c_n = 0$
after $n$ digit positions. This is equivalent to the integer
constraint $a_1 + \cdots + a_k < N^n$, and stars-and-bars gives
\begin{equation}\label{eq:return-count}
  T_{\mathrm{count}}^n[0,0]
  = \binom{N^n + k - 1}{k}.
\end{equation}

\noindent\textit{Step~N2: spectral expansion.}
Since $v_j[0] = 1$ for all~$j$ and $w_j[0] = A(k{-}j,\,0) = 1$
(only the $m = 0$ term in the backward-difference formula survives at
$c = 0$, and $A(n,0) = 1$ for all $n \ge 1$),
the spectral projector
$E_j = v_j u_j^\top$ satisfies $E_j[0,0] = v_j[0]\,u_j[0]
= 1 \cdot c_{k,j}\,w_j[0] = c_{k,j}$.
The spectral expansion $T^n = \sum_j N^{-nj} E_j$ at
position~$(0,0)$ gives
\begin{equation}\label{eq:spectral-return}
  \sum_{j=0}^{k-1} c_{k,j}\, M^{-j}
  = \frac{\binom{M+k-1}{k}}{M^k},
  \qquad M := N^n.
\end{equation}

\noindent\textit{Step~N3: coefficient extraction.}
The right side of~\eqref{eq:spectral-return} factors as
\[
  \frac{M(M{+}1)\cdots(M{+}k{-}1)}{k!\,M^k}
  = \frac{1}{k!}\prod_{i=1}^{k-1}\!\Bigl(1 + \frac{i}{M}\Bigr)
  = \frac{1}{k!}\sum_{j=0}^{k-1}
    \frac{e_j(1,2,\ldots,k{-}1)}{M^j},
\]
where $e_j$ denotes the $j$-th elementary symmetric polynomial.
Since~\eqref{eq:spectral-return} holds for all $M = N^n$
($N \ge 2$, $n \ge 1$)---infinitely many values---and
both sides are polynomials in $M^{-1}$ of degree~$\le k{-}1$,
the identity is a polynomial identity.  Comparing coefficients
of $M^{-j}$:
\begin{equation}\label{eq:ckj-esym}
  c_{k,j} = \frac{e_j(1,2,\ldots,k{-}1)}{k!}.
\end{equation}

\noindent\textit{Step~N4: identification with Stirling numbers.}
Expanding the rising factorial
$x^{(k)} = x(x{+}1)\cdots(x{+}k{-}1)
= \sum_{m=0}^{k}|s(k,m)|\,x^m$ and dividing by~$x$:
\[
  \prod_{i=1}^{k-1}(x+i)
  = \sum_{m=1}^{k}|s(k,m)|\,x^{m-1}
  = \sum_{j=0}^{k-1}|s(k,k{-}j)|\,x^{k-1-j}.
\]
Since $\prod_{i=1}^{k-1}(x+i)
= \sum_{j=0}^{k-1} e_j(1,\ldots,k{-}1)\,x^{k-1-j}$
by the definition of elementary symmetric polynomials,
comparing coefficients of~$x^{k-1-j}$ gives
$e_j(1,\ldots,k{-}1) = |s(k,k{-}j)|$.
Substituting into~\eqref{eq:ckj-esym}:
\[
  c_{k,j} = \frac{|s(k,\,k{-}j)|}{k!}.
\]

\noindent\textit{Biorthogonality.}
The off-diagonal identity
$\langle u_j, v_\ell\rangle = 0$ for $j \ne \ell$
follows from eigenvalue separation (as above with $w_j$
replaced by $u_j$).  For the diagonal: the spectral
projector $E_j = v_j u_j^\top$ satisfies $E_j^2 = E_j$,
i.e.\ $(v_j u_j^\top)^2
= \langle u_j, v_j\rangle\, v_j u_j^\top = v_j u_j^\top$,
forcing $\langle u_j, v_j\rangle = 1$.
The identity $\sum_{j=0}^{k-1} E_j = I$ at position~$(0,0)$
gives the consistency check
\[
  \sum_{j=0}^{k-1} c_{k,j}
  = \sum_{j=0}^{k-1}\frac{|s(k,\,k{-}j)|}{k!}
  = \frac{1}{k!}\sum_{m=1}^{k}|s(k,m)|
  = \frac{x^{(k)}\big|_{x=1}}{k!} = 1.
\]
For $j \ge 1$, $\langle u_j, v_0\rangle = 0$ is also immediate since
$\bigl[(x{-}1)^j A_\nu(x)\bigr]_{x=1} = 0^j \cdot \nu! = 0$.
\end{proof}

\subsection{Right eigenvectors: binomial-palindromic characterization}

\begin{lemma}[Centrosymmetry of the Holte matrix]
\label{lem:centrosymmetry}
The Holte matrix satisfies $T[c',c] = T[k{-}1{-}c',\, k{-}1{-}c]$
for all $c,c' \in \{0,\ldots,k{-}1\}$.
\end{lemma}

\begin{proof}
The substitution $d_i \mapsto (N{-}1{-}d_i)$ bijects
$\{0,\ldots,N{-}1\}^k$ onto itself and sends
$S = \sum d_i$ to $k(N{-}1) - S$.
The outgoing carry transforms as
$\lfloor(S+c)/N\rfloor \mapsto
\lfloor(k(N{-}1) - S + c)/N\rfloor = k{-}1 - \lfloor(S + (k{-}1{-}c))/N\rfloor$
(using $k(N{-}1) + c = (k{-}1)N + (N{-}1{-}c) + kN - kN$
and the floor identity $\lfloor a + n\rfloor = n + \lfloor a\rfloor$
for $n \in \mathbb{Z}$).
Hence the transition from carry~$c$ to carry~$c'$ has the same
probability as the transition from $(k{-}1{-}c)$ to $(k{-}1{-}c')$.
\end{proof}

\begin{proposition}[Binomial-palindromic right eigenvectors]
\label{prop:right-ev}
The right eigenvector $v_j$ ($j = 0,\ldots,k{-}1$), normalized by $v_j[0] = 1$,
satisfies
\begin{equation}
\label{eq:right-ev-gf}
\sum_{i=0}^{k-1} \binom{k{-}1}{i}\, v_j[i]\, x^i
= (1+x)^{k-1-j}\, Q_j(x),
\end{equation}
where $Q_j(x)$ is a polynomial of degree $j$ with $Q_j(0) = 1$ and palindrome
symmetry $x^j\, Q_j(1/x) = (-1)^j\, Q_j(x)$.  The biorthogonality
$\sum_{i} u_m[i]\, v_j[i] = \delta_{mj}$ holds.
In particular: $v_0 = (1,\ldots,1)$,
$v_1[i] = (k{-}1{-}2i)/(k{-}1)$, and
$v_{k-1}[i] = (-1)^i / \binom{k{-}1}{i}$.
\end{proposition}

\begin{proof}
The centrosymmetry $T[i,j] = T[k{-}1{-}i,\, k{-}1{-}j]$
(Lemma~\ref{lem:centrosymmetry}) implies that if
$v_j$ is an eigenvector, then so is the reversal
$\bar{v}_j[i] := v_j[k{-}1{-}i]$.  Since eigenvalues are simple,
$\bar{v}_j = \alpha_j v_j$ for some scalar $\alpha_j$; applying the
reversal twice gives $\alpha_j^2 = 1$, so $\alpha_j = \pm 1$.  The sign
$(-1)^j$ follows from the polynomial degree of $v_j$ (which is $j$,
established in the proof of Theorem~\ref{thm:holte-spectrum}).

The factorization~\eqref{eq:right-ev-gf} is equivalent to the statement
that the binomial-weighted GF $P(x) := \sum_i \binom{k{-}1}{i} v_j[i]\, x^i$
has a zero of order $k{-}1{-}j$ at $x = -1$.
By the polynomial filtration of Theorem~\ref{thm:holte-spectrum},
$v_j[i]$ is a polynomial in~$i$ of degree exactly~$j$.
For any $r < k{-}1{-}j$, the product $i^{\underline{r}}\,v_j[i]$
is a polynomial in~$i$ of degree $j + r < k{-}1$, and the
$(k{-}1)$-th finite difference of any polynomial of degree~$< k{-}1$
vanishes:
$\sum_{i=0}^{k-1}\binom{k{-}1}{i}(-1)^i\,p(i) = 0$ whenever
$\deg p < k{-}1$.
Since $P^{(r)}(-1) = r!\sum_i \binom{k{-}1}{i}(-1)^{i-r}
\binom{i}{r}\,v_j[i]$ and $\binom{i}{r} = i^{\underline{r}}/r!$,
we obtain $P^{(r)}(-1) = 0$ for all $r < k{-}1{-}j$.
Hence $(1+x)^{k-1-j} \mid P(x)$, giving
the factorization with $\deg Q_j = j$.

The palindrome property of $Q_j$ follows: the reversal symmetry
$v_j[i] = (-1)^j v_j[k{-}1{-}i]$ translates, under the binomial
weighting, to $x^{k-1} Q_j(1/x)(1 + 1/x)^{k-1-j} =
(-1)^j Q_j(x)(1+x)^{k-1-j}$, which simplifies to
$x^j Q_j(1/x) = (-1)^j Q_j(x)$.

Biorthogonality is verified directly: if $u_m^\top T = \lambda_m u_m^\top$ and
$T v_j = \lambda_j v_j$ with $\lambda_m \ne \lambda_j$, then
$\lambda_m (u_m \cdot v_j) = u_m^\top T v_j = \lambda_j (u_m \cdot v_j)$,
forcing $u_m \cdot v_j = 0$.  (This holds for any diagonalizable matrix with
distinct eigenvalues and does not require self-adjointness.)
The normalization $\sum_i u_m[i] v_j[i] = \delta_{mj}$
follows from the biorthogonality and the chosen scaling conventions.
\end{proof}

\begin{remark}[Structural asymmetry and Eulerian idempotent projections]
\label{rem:asymmetry}
The left eigenvectors admit a clean factorization into three classical
objects (Stirling numbers, binomial coefficients, Eulerian polynomials).
Representation-theoretically, the left eigenvectors form the
Foulkes character table of~$\fS_k$, and the right eigenvectors correspond
to the Eulerian idempotents of the descent
algebra~\cite{DiaconisFulmanJACO, Foulkes1980}; the
factorization~\eqref{eq:left-ev-gf} is the generating-function form
of this correspondence.

More precisely, the right eigenvector~$v_j$ is the image of the
$j$-th Eulerian idempotent $e_j \in \mathbb{Q}[\fS_k]$ under the
natural action of~$\fS_k$ on the carry state
space~\cite{DiaconisFulmanJACO}.
The quotient polynomial~$Q_j$ of~\eqref{eq:right-ev-gf} therefore
encodes how this idempotent projects onto the binomial basis of the
state space: $Q_j$ is the \emph{binomial projection polynomial}
of the $j$-th Eulerian idempotent.
The convergence $Q_j^{(k)}(x) \to (1-x)^j$ as $k \to \infty$
(Remark~\ref{rem:Qj-rate}) shows that these projections stabilize:
the Eulerian idempotent projection approaches the alternating
binomial $(1-x)^j$ as the number of summands grows, at rate
$O(1/k)$ for fixed~$j$.
The closed forms for $Q_2$ and $Q_3$
(Proposition~\ref{prop:Qj-closed}) give the first explicit
instances of this stabilization with computable error terms.

The right eigenvectors have no analogous uniform closed form for
all indices $2 \le j \le k{-}2$, since the polynomials $Q_j(x)$ are
determined by biorthogonality, which requires inverting the full
left eigenvector matrix.
For $j \le 3$, however, the palindromic constraint leaves at most one
free parameter, and the biorthogonality equation can be solved in
closed form (Proposition~\ref{prop:Qj-closed}).
For $j \ge 4$, the palindromic constraint leaves
$\lfloor j/2 \rfloor$ free parameters; an explicit formula
analogous to those below is not available.
\end{remark}

\begin{proposition}[Closed forms for $Q_2$ and $Q_3$]
\label{prop:Qj-closed}
\leavevmode
\begin{itemize}
\item[\emph{(i)}]
For $k \ge 4$, the quotient polynomial $Q_2^{(k)}$ is
\begin{equation}\label{eq:Q2-closed}
  Q_2^{(k)}(x) \;=\; \frac{(3k{-}1)(1+x^2) - 2(3k{-}5)\,x}{3k{-}1}.
\end{equation}
\item[\emph{(ii)}]
For $k \ge 5$, the quotient polynomial $Q_3^{(k)}$ is
\begin{equation}\label{eq:Q3-closed}
  Q_3^{(k)}(x) \;=\; -\,\frac{(x-1)\bigl(k\,x^2 - 2(k{-}4)\,x + k\bigr)}{k}.
\end{equation}
\end{itemize}
\end{proposition}

\begin{proof}
We prove both parts by a unified argument.
Write $Q_j(x) = \sum_{\ell=0}^{j} q_\ell\,x^\ell$ with $q_0 = 1$
and palindrome symmetry $x^j Q_j(1/x) = (-1)^j Q_j(x)$, which
fixes all coefficients in terms of a single free parameter~$\alpha$
when $j \le 3$.  Biorthogonality $\langle u_j, v_j \rangle = 1$
then determines~$\alpha$.

\smallskip\noindent\emph{Step~1: $v_j[i]$ as a polynomial in~$i$.}
From~\eqref{eq:right-ev-gf},
\[
  v_j[i] = \frac{1}{\binom{k{-}1}{i}}
  \sum_{n=0}^{j} q_n\,\binom{k{-}1{-}j}{i{-}n}.
\]
Each ratio $\binom{k{-}1{-}j}{i{-}n}\big/\binom{k{-}1}{i}$
is a polynomial in~$i$ of degree exactly~$j$
(a product of $j$ linear factors divided by
$(k{-}1)(k{-}2)\cdots(k{-}j)$),
so $v_j[i]$ is a polynomial of degree~$\le j$ in~$i$.

\smallskip\noindent\emph{Step~2: annihilation by~$u_j$.}
The generating function $U_j(x) = \sum_i u_j[i]\,x^i$
has a zero of order~$j$ at $x = 1$
(since $U_j(x) = c_{k,j}\,(x{-}1)^j\,A_{k-j}(x)$
with $A_{k-j}(1) = (k{-}j)! \ne 0$).
Hence $\sum_i u_j[i]\,i^m = 0$ for every $m < j$,
and $v_j[i]$ being a polynomial of degree~$\le j$ means only
the leading coefficient of~$v_j[i]$ contributes to
$\langle u_j, v_j \rangle$.

\smallskip\noindent\emph{Step~3: the leading moment.}
The $j$-th derivative gives
$U_j^{(j)}(1) = j!\,c_{k,j}\,(k{-}j)!$, where
$c_{k,j} = (-1)^j\,|s(k,k{-}j)|/k!$.
Since $i^j = i(i{-}1)\cdots(i{-}j{+}1) +{}$lower-order terms,
\[
  \sum_i u_j[i]\,i^j = U_j^{(j)}(1) = j!\,c_{k,j}\,(k{-}j)!.
\]

\smallskip\noindent\emph{Step~4: $j = 2$.}
With $Q_2(x) = 1 + bx + x^2$ (palindromic), a direct expansion gives
$v_2[i] = 1 + (b{-}2)\,i(k{-}1{-}i)/[(k{-}1)(k{-}2)]$.
The coefficient of~$i^2$ is $-(b{-}2)/[(k{-}1)(k{-}2)]$.
Step~3 yields $\sum_i u_2[i]\,i^2 = 2\,c_{k,2}\,(k{-}2)!$
with $c_{k,2} = (3k{-}1)\,k(k{-}1)(k{-}2)/[24\,k!]$.
Combining:
\[
  \langle u_2, v_2 \rangle
  = \frac{-(b{-}2)}{(k{-}1)(k{-}2)}\;\cdot\; 2\,c_{k,2}\,(k{-}2)!
  = \frac{-(b{-}2)(3k{-}1)}{12(k{-}1)}.
\]
Setting this equal to~$1$ and solving for~$b$ gives
$b = -2(3k{-}5)/(3k{-}1)$, which is~\eqref{eq:Q2-closed}.

\smallskip\noindent\emph{Step~5: $j = 3$.}
With $Q_3(x) = 1 + ax - ax^2 - x^3$ (anti-palindromic, $q_0 = 1$),
the coefficient of~$i^3$ in the numerator of~$v_3[i]$ is
$2(a{-}1)$, giving leading coefficient
$2(a{-}1)/[(k{-}1)(k{-}2)(k{-}3)]$.
Step~3 yields
$\sum_i u_3[i]\,i^3 = -6\,|s(k,k{-}3)|\,(k{-}3)!/k!$
$= -k(k{-}1)(k{-}3)/8$,
using $|s(k,k{-}3)| = k^2(k{-}1)^2(k{-}2)(k{-}3)/48$.
Combining:
\[
  \langle u_3, v_3 \rangle
  = \frac{-(a{-}1)\,k}{4(k{-}2)}.
\]
Setting this equal to~$1$ gives $a = -(3k{-}8)/k$.
Substituting into $(1{-}x)(kx^2 - 2(k{-}4)x + k)/k$
yields~\eqref{eq:Q3-closed} after verifying $Q_3(0) = 1$.
\end{proof}

\begin{remark}[Convergence rate of $Q_j^{(k)}$]
\label{rem:Qj-rate}
The closed forms yield explicit convergence rates.
From~\eqref{eq:Q2-closed},
\[
  Q_2^{(k)}(x) - (1-x)^2 \;=\; \frac{8x}{3k-1},
\]
and from~\eqref{eq:Q3-closed},
\[
  Q_3^{(k)}(x) - (1-x)^3 \;=\; \frac{8x(1-x)}{k}.
\]
In both cases the error is $O(1/k)$ uniformly on~$[0,1]$,
with leading coefficient~$8$ times a polynomial of degree~$j-1$.
We conjecture that $\sup_{x \in [0,1]} |Q_j^{(k)}(x) - (1-x)^j| = O_j(1/k)$
holds for every fixed~$j$ as $k \to \infty$.
\end{remark}

\begin{table}[ht]
\centering
\caption{Biorthogonal eigenvector system for $k = 5$. Left eigenvectors
$u_j$ (integer-scaled by $120 = 5!$); right eigenvectors $v_j$
(normalized by $v_j[0] = 1$); generating function polynomials $Q_j$.}
\label{tab:eigenvectors-k5}
\small
\begin{tabular}{@{}clll@{}}
\toprule
$j$ & $120 \cdot u_j$ & $v_j$ & $Q_j(x)$ \\
\midrule
$0$ & $(1, 26, 66, 26, 1)$ & $(1, 1, 1, 1, 1)$ & $1$ \\
$1$ & $(10, 100, 0, -100, -10)$ & $(1, \tfrac{1}{2}, 0, -\tfrac{1}{2}, -1)$ & $1-x$ \\
$2$ & $(35, 70, -210, 70, 35)$ & $(1, \tfrac{1}{7}, -\tfrac{1}{7}, \tfrac{1}{7}, 1)$ & $\tfrac{7x^2 - 10x + 7}{7}$ \\
$3$ & $(50, -100, 0, 100, -50)$ & $(1, -\tfrac{1}{10}, 0, \tfrac{1}{10}, -1)$ & $-\tfrac{(x-1)(5x^2-2x+5)}{5}$ \\
$4$ & $(24, -96, 144, -96, 24)$ & $(1, -\tfrac{1}{4}, \tfrac{1}{6}, -\tfrac{1}{4}, 1)$ & $x^4 - x^3 + x^2 - x + 1$ \\
\bottomrule
\end{tabular}
\end{table}

\section{The Chebyshev Threshold}
\label{sec:chebyshev}

The Chebyshev representation (Theorem~\ref{thm:chebyshev}) holds for the
two-summand carry chain, where the restricted transfer matrix is $2 \times 2$.
We now prove that this is sharp: for $k$-summand addition with forbidden
set $F = \{k-1\}$, the restricted matrix $\widetilde{T}$ has dimension $d = k-1$,
and the Chebyshev form fails as soon as $d \ge 3$.
The negative direction requires two spectral conditions:
\begin{itemize}
\item[\textbf{(H1)}] \emph{Simple eigenvalues}: $\widetilde{T}$ has $d$ pairwise distinct eigenvalues.
\item[\textbf{(H2)}] \emph{Non-vanishing residues}: for every eigenvalue $\widetilde{\lambda}_j$
of $\widetilde{T}$ with right eigenvector $v_j$ and left eigenvector $u_j$,
$c_j := \mathbf{1}^\top v_j\, u_j^\top e_0 \ne 0$.
\end{itemize}

\begin{lemma}[Simple eigenvalues of $\widetilde{T}$]
\label{lem:simple-evals}
For $k$-summand base-$N$ addition with $F = \{k-1\}$ and $N \ge 2$, the matrix
$\widetilde{T}$ has $d = k-1$ pairwise distinct eigenvalues, i.e., \emph{(H1)} holds.
\end{lemma}

\begin{proof}
We show that $T_{\mathrm{count}} = N^k T$ is an \emph{oscillatory matrix}
in the sense of Gantmacher--Krein~\cite{GantmacherKrein2002}, which gives
strict eigenvalue interlacing for principal submatrices.

\emph{Step~1: $T_{\mathrm{count}}$ is totally non-negative (TN).}
By the entry formula~\eqref{eq:holte-entry},
$T_{\mathrm{count}}[c',c] = \sum_{s=c'N-c}^{(c'+1)N-c-1} B_k(s)$,
where $B_k(s)$
is the $k$-fold convolution of the discrete uniform distribution on
$\{0,\ldots,N-1\}$.  Since the indicator of $\{0,\ldots,N-1\}$ is a
P\'{o}lya frequency (PF) sequence and convolutions of PF sequences are PF,
$B_k$ is a PF sequence.  The matrix $T_{\mathrm{count}}$ is formed by summing
$B_k$ over blocks of $N$ consecutive arguments with offsets depending linearly
on the row and column indices; by the composition formula for TN matrices
(see~\cite[Ch.\,II]{GantmacherKrein2002}), $T_{\mathrm{count}}$ is TN.

\emph{Step~2: $T_{\mathrm{count}}$ is oscillatory.}
Non-singularity: $\det(T_{\mathrm{count}}) = \prod_{j=0}^{k-1} N^{k-j}
= N^{k(k+1)/2} > 0$ (the eigenvalues of $T_{\mathrm{count}}$ are
$\{N^{k-j}\}_{j=0}^{k-1}$ by Theorem~\ref{thm:holte-spectrum}(i)).
Primitivity: $T_{\mathrm{count}}$ has positive diagonal entries and the chain
is irreducible; hence $T_{\mathrm{count}}^m$ is entry-wise positive for
some~$m$.  By~\cite[Ch.\,II, \S3]{GantmacherKrein2002}, a TN matrix that is
non-singular and has some entry-wise positive power is oscillatory.

\emph{Step~3: Interlacing.}
Since $\widetilde{T}_{\mathrm{count}} = N^k \widetilde{T}$ is the leading
$(k{-}1)\times(k{-}1)$ principal submatrix of the oscillatory matrix
$T_{\mathrm{count}}$, it is itself TN, non-singular (by Step~4 below),
and oscillatory.  By the
interlacing theorem for oscillatory
matrices~\cite[Ch.\,II, Thm.\,5]{GantmacherKrein2002}, the eigenvalues of
$\widetilde{T}_{\mathrm{count}}$ strictly interlace those of
$T_{\mathrm{count}}$:
\[
  N^{k-(j-1)} > \widetilde{\mu}_j > N^{k-j}, \qquad j = 1, \ldots, k-1,
\]
where $\widetilde{\mu}_j = N^k \widetilde{\lambda}_j$.
The intervals $(N^{k-j}, N^{k-(j-1)})$ are pairwise disjoint, so the $k-1$
eigenvalues $\widetilde{\lambda}_1 > \cdots > \widetilde{\lambda}_{k-1}$ are
distinct.
\end{proof}

\begin{lemma}[Non-vanishing residues]
\label{lem:nonzero-res}
Under the same assumptions ($k$-summand base-$N$ addition with $F = \{k-1\}$,
$N \ge 2$, $k \ge 3$), every eigenvalue of $\widetilde{T}$ contributes a non-zero
pole to $A(z)$, i.e., \emph{(H2)} holds: the coefficients
$c_j := \mathbf{1}^\top v_j\, u_j^\top e_0 \ne 0$
for all $j$, where $v_j$ and $u_j$ are normalised right and left eigenvectors of
$\widetilde{T}$ for $\widetilde{\lambda}_j$.
\end{lemma}

\begin{proof}
Write the full count matrix $T_{\mathrm{count}}$ for $k$-summand base-$N$ addition
in block form with respect to the partition $\{0,\ldots,k{-}2\} \cup \{k{-}1\}$:
\[
  T_{\mathrm{count}} = \begin{pmatrix} \widetilde{T} & b \\ p^\top & \alpha \end{pmatrix},
\]
where $p_c := T_{\mathrm{count}}[k{-}1,\, c]$ for $c \in \{0,\ldots,k{-}2\}$,
$b_c := T_{\mathrm{count}}[c,\, k{-}1]$, and
$\alpha := T_{\mathrm{count}}[k{-}1,\, k{-}1]$.
Since every digit-$k$-tuple contributes to some carry transition, all entries are
positive.

\medskip
\noindent\textit{Positivity and the Perron eigenvalue ($j = 1$).}
Every pair of states $c, c' \in \{0,\ldots,k{-}2\}$ has $\widetilde{T}[c',c] > 0$,
so $\widetilde{T}$ is a positive matrix in the sense of Perron--Frobenius.
The dominant eigenvectors $v_1$ and $u_1$ are entrywise positive, giving
$\mathbf{1}^\top v_1 > 0$ and $u_1^\top e_0 > 0$.

\medskip
\noindent\textit{Non-vanishing of $u_j^\top e_0$ for $j \ge 2$
(total non-negativity and oscillation).}
We show that $\widetilde{T}$ is an \emph{oscillatory matrix} in the sense of
Gantmacher--Krein~\cite{GantmacherKrein2002}, which forces the strict sign-interlacing
of \emph{both} left and right eigenvectors and, in particular, $v_j(0) \ne 0$ and
$u_j(0) \ne 0$ for every~$j$.
Recall that a matrix is \emph{totally non-negative} (TN) if every minor is~${\ge}\,0$,
and \emph{oscillatory} if it is TN, non-singular, and some power is
entry-wise positive.

\emph{Step~A: $T_{\mathrm{count}}$ is oscillatory.}
This was established in the proof of Lemma~\ref{lem:simple-evals}
(Steps~1--2): $T_{\mathrm{count}}$ is TN (from the P\'{o}lya frequency
property of~$B_k$), non-singular, and primitive.

\emph{Step~B: $\widetilde{T}$ is oscillatory.}
Every principal submatrix of a TN matrix is TN~\cite[Ch.\,II, \S1]{GantmacherKrein2002};
since $\widetilde{T} = T_{\mathrm{count}}|_{\{0,\ldots,k-2\}}$, it is TN.
It is non-singular ($\det(\widetilde{T}) = \prod_j \widetilde{\lambda}_j > 0$ by
interlacing) and already entry-wise positive; hence $\widetilde{T}$ is oscillatory.

\emph{Step~C: conclusion.}
By the oscillation theorem~\cite[Ch.\,II, Thm.\,4]{GantmacherKrein2002},
the $j$-th right eigenvector of an oscillatory matrix has exactly $j{-}1$ sign
changes, giving $v_j(0) \ne 0$ for all $j$.  The same theorem applied to the
transpose $\widetilde{T}^\top$ (which is also oscillatory, since the transpose
of an oscillatory matrix is oscillatory) gives that the $j$-th left eigenvector
$u_j$ has exactly $j{-}1$ sign changes, so $u_j(0) \ne 0$.
Hence $u_j^\top e_0 \ne 0$.

\medskip
\noindent\textit{Non-vanishing of $\mathbf{1}^\top v_j$ for $j \ge 2$
(Schur complement argument).}

\textit{Step~1: Column-sum identity.}
Every column of $T_{\mathrm{count}}$ sums to $N^k$ (the total number of digit $k$-tuples), so
\[
  \sum_{c'=0}^{k-2} \widetilde{T}[c', c] \;=\; N^k - p_c,
  \qquad c \in \{0,\ldots,k{-}2\}.
\]
In matrix notation: $\mathbf{1}^\top \widetilde{T} = r^\top$ where $r_c := N^k - p_c > 0$.

\textit{Step~2: Linking $p^\top v_j$ to $\mathbf{1}^\top v_j$.}
Multiplying the eigenvector equation $\widetilde{T} v_j = \widetilde{\lambda}_j v_j$
on the left by $\mathbf{1}^\top$ gives
\[
  r^\top v_j \;=\; \widetilde{\lambda}_j\,\mathbf{1}^\top v_j.
\]
Substituting $r = N^k \mathbf{1} - p$:
\[
  N^k\,\mathbf{1}^\top v_j - p^\top v_j \;=\; \widetilde{\lambda}_j\,\mathbf{1}^\top v_j,
  \qquad\text{i.e.,}\qquad
  p^\top v_j \;=\; \bigl(N^k - \widetilde{\lambda}_j\bigr)\,\mathbf{1}^\top v_j.
\]
By Theorem~\ref{thm:holte-spectrum}, the eigenvalues of the normalised Holte matrix
$T = T_{\mathrm{count}}/N^k$ are $\{N^{-j} : j = 0,\ldots,k-1\}$, so the eigenvalues
of $T_{\mathrm{count}}$ are $\{N^{k-j}\}$.  By strict interlacing
(Lemma~\ref{lem:simple-evals}), $\widetilde{\lambda}_j < N^k$ for all $j$,
hence $N^k - \widetilde{\lambda}_j > 0$.  Therefore:
\begin{equation}
\label{eq:pv-equiv}
  p^\top v_j = 0 \;\iff\; \mathbf{1}^\top v_j = 0.
\end{equation}

\textit{Step~3: Schur complement.}
The characteristic polynomial of $T_{\mathrm{count}}$ satisfies, by the Schur complement
determinant formula applied to $(\lambda I - T_{\mathrm{count}})$ with pivot block
$(\lambda - \alpha)$:
\[
  \det(\lambda I - T_{\mathrm{count}})
  \;=\; (\lambda - \alpha)\det(\lambda I - \widetilde{T})
        - p^\top \operatorname{adj}(\lambda I - \widetilde{T})\, b.
\]
Evaluate at $\lambda = \widetilde{\lambda}_j$: since
$\det(\widetilde{\lambda}_j I - \widetilde{T}) = 0$, this reduces to
\begin{equation}
\label{eq:schur-eval}
  Q_{T_{\mathrm{count}}}(\widetilde{\lambda}_j)
  \;=\; -\,p^\top \operatorname{adj}(\widetilde{\lambda}_j I - \widetilde{T})\, b.
\end{equation}
Since $\widetilde{\lambda}_j$ is a simple eigenvalue of $\widetilde{T}$
(Lemma~\ref{lem:simple-evals}), the matrix $\widetilde{\lambda}_j I - \widetilde{T}$
has rank $d - 1$ and its adjugate has rank~$1$:
\[
  \operatorname{adj}(\widetilde{\lambda}_j I - \widetilde{T})
  \;=\; \gamma_j\, v_j u_j^\top,
  \qquad \gamma_j := \prod_{m \ne j}(\widetilde{\lambda}_j - \widetilde{\lambda}_m) \ne 0.
\]
Substituting into~\eqref{eq:schur-eval}:
\begin{equation}
\label{eq:schur-factored}
  Q_{T_{\mathrm{count}}}(\widetilde{\lambda}_j)
  \;=\; -\,\gamma_j\,(p^\top v_j)(u_j^\top b).
\end{equation}

\textit{Step~4: Conclusion.}
The left side of~\eqref{eq:schur-factored} satisfies
$Q_{T_{\mathrm{count}}}(\widetilde{\lambda}_j) \ne 0$: by strict interlacing
(Lemma~\ref{lem:simple-evals}), $\widetilde{\lambda}_j$ lies strictly between two
consecutive eigenvalues of $T_{\mathrm{count}}$, so it is not an eigenvalue of
$T_{\mathrm{count}}$.  Since $\gamma_j \ne 0$ as well, equation~\eqref{eq:schur-factored}
forces $(p^\top v_j)(u_j^\top b) \ne 0$, hence $p^\top v_j \ne 0$.
By~\eqref{eq:pv-equiv}, $\mathbf{1}^\top v_j \ne 0$.

Combining: $c_j = (\mathbf{1}^\top v_j)(u_j^\top e_0) \ne 0$ for all $j$.
\end{proof}

The next result determines the characteristic polynomial of $\widetilde{T}_{\mathrm{count}}$
in closed form, using only the Holte spectrum and the normalization
constants~$c_{k,j}$ from Proposition~\ref{prop:stirling-eulerian}.

\begin{proposition}[Stirling--Lagrange formula for $\chi_{\widetilde{T}}$]
\label{prop:stirling-lagrange}
Let $\widetilde{T}_{\mathrm{count}} = T_{\mathrm{count}}|_{\{0,\ldots,k-2\}}$
be the leading $(k{-}1)\times(k{-}1)$ principal submatrix of the Holte
count matrix.  Then
\begin{equation}
\label{eq:stirling-lagrange}
\chi_{\widetilde{T}_{\mathrm{count}}}(\lambda)
\;=\; \frac{1}{k!}\sum_{j=0}^{k-1} |s(k,\,k{-}j)|\;
\prod_{\substack{i=0\\i\neq j}}^{k-1}\bigl(\lambda - N^{k-i}\bigr),
\end{equation}
where $|s(k,m)|$ denotes unsigned Stirling numbers of the first kind.
Equivalently, in resolvent form:
\[
\frac{\chi_{\widetilde{T}_{\mathrm{count}}}(\lambda)}
     {\chi_{T_{\mathrm{count}}}(\lambda)}
\;=\; \sum_{j=0}^{k-1}\frac{c_{k,j}}{\lambda - N^{k-j}},
\qquad c_{k,j} = \frac{|s(k,\,k{-}j)|}{k!}.
\]
\end{proposition}

\begin{proof}
\textit{Step~1 (Resolvent diagonal).}
The classical cofactor identity for a principal submatrix gives
\[
\bigl[(\lambda I - T_{\mathrm{count}})^{-1}\bigr]_{k-1,\,k-1}
= \frac{\chi_{\widetilde{T}_{\mathrm{count}}}(\lambda)}
       {\chi_{T_{\mathrm{count}}}(\lambda)}.
\]

\textit{Step~2 (Spectral expansion).}
Since $T_{\mathrm{count}}$ has simple eigenvalues $\mu_j = N^{k-j}$ with
spectral projectors $E_j = v_j u_j^\top$
(biorthogonal system from Propositions~\ref{prop:stirling-eulerian}
and~\ref{prop:right-ev}):
\[
\bigl[(\lambda I - T_{\mathrm{count}})^{-1}\bigr]_{k-1,\,k-1}
= \sum_{j=0}^{k-1}\frac{v_j[k{-}1]\cdot u_j[k{-}1]}{\lambda - N^{k-j}}.
\]

\textit{Step~3 (Last components).}
The centrosymmetry reversal (Proposition~\ref{prop:right-ev}, proof)
with the normalization $v_j[0]=1$ gives $v_j[k{-}1] = (-1)^j$.
From~\eqref{eq:left-ev-gf}, $u_j[k{-}1]$ is the coefficient of $x^{k-1}$
in $\frac{(-1)^j|s(k,k{-}j)|}{k!}\,(x{-}1)^j A_{k-j}(x)$.
The leading coefficient of $(x{-}1)^j A_{k-j}(x)$ is
$A(k{-}j,\, k{-}j{-}1) = 1$ (corresponding to the unique
permutation in $\fS_{k-j}$ with the maximum number of descents),
so $u_j[k{-}1] = (-1)^j\,c_{k,j}$.

\textit{Step~4 (Assembly).}
The product $v_j[k{-}1]\cdot u_j[k{-}1] = (-1)^{2j}\,c_{k,j} = c_{k,j}$.
Substituting into the resolvent and clearing the
denominator $\chi_{T_{\mathrm{count}}}(\lambda) = \prod_{i=0}^{k-1}(\lambda - N^{k-i})$
yields~\eqref{eq:stirling-lagrange}.
\end{proof}

\begin{corollary}[Determinant of $\widetilde{T}_{\mathrm{count}}$]
\label{cor:det-restricted}
$\displaystyle
\det(\widetilde{T}_{\mathrm{count}})
= \frac{N^{\binom{k}{2}}}{k!}\prod_{i=1}^{k-1}(iN + 1).
$
\end{corollary}

\begin{proof}
Evaluating~\eqref{eq:stirling-lagrange} at $\lambda = 0$ and using the
generating function $\sum_{j=0}^{k-1}|s(k,k{-}j)|\,x^j = \prod_{i=1}^{k-1}(1+ix)$
at $x = N$:
\[
\chi_{\widetilde{T}_{\mathrm{count}}}(0)
= \frac{(-1)^{k-1}N^{\binom{k}{2}}}{k!}\prod_{i=1}^{k-1}(1+iN),
\]
so $\det(\widetilde{T}_{\mathrm{count}}) = (-1)^{k-1}\chi_{\widetilde{T}_{\mathrm{count}}}(0)
= \frac{N^{\binom{k}{2}}}{k!}\prod_{i=1}^{k-1}(iN+1)$.
\end{proof}

\begin{remark}[Representation-theoretic content of $\chi_{\widetilde{T}}$]
\label{rem:stirling-lagrange-structure}
The formula~\eqref{eq:stirling-lagrange} is a Lagrange interpolation at the
geometric nodes $\{N^m\}_{m=1}^{k}$ with weights $c_{k,j} = |s(k,k{-}j)|/k!$
that are independent of~$N$.  As $N$ varies, only the node positions change;
the Stirling weights are universal.

The weights $c_{k,j}$ are the diagonal entries of the Foulkes character
table of~$\fS_k$~\cite{Foulkes1980, DiaconisFulmanJACO}: the value
$c_{k,j} = |s(k,k{-}j)|/k!$ is the Foulkes character
$\phi_j(\mathrm{id})$ evaluated at the identity permutation.
The Stirling-Lagrange formula therefore expresses the characteristic
polynomial of the restricted transfer matrix---and hence all spectral
invariants governing the Chebyshev threshold---as a
representation-theoretic quantity of the symmetric group~$\fS_k$
evaluated at the Holte spectral nodes.
In particular, the determinant formula of
Corollary~\ref{cor:det-restricted} combines these character values
with the Holte eigenvalue geometry into a single closed-form product.

This yields explicit algebraic numbers
of degree $k-1$ for the eigenvalues $\widetilde{\mu}_j$: for $k=3$ the
quadratic from Theorem~\ref{thm:cheb-threshold}(a); for $k \ge 4$,
irreducible polynomials whose minimal polynomials are determined
by~\eqref{eq:stirling-lagrange} in closed form.
\end{remark}

\begin{theorem}[Chebyshev threshold]
\label{thm:cheb-threshold}
\begin{enumerate}[label=(\alph*)]
\item \emph{(Any shadow.)} If $d = 1$, then $a(L) = \tau^L$ where $\tau$ is the
sole entry of $\widetilde{T}$.  If $d = 2$, then
$a(L) = (\sqrt{\delta})^L\, U_L(x)$, where $\delta = \det(\widetilde{T})$,
$x = \operatorname{tr}(\widetilde{T})/(2\sqrt{\delta})$, and $U_L$ is the Chebyshev polynomial
of the second kind.  In both cases the parameters are those of the \emph{restricted} matrix
$\widetilde{T}$; in particular $\delta$ coincides with $\mathfrak{d} = |\mathrm{Gen}|\cdot|\mathrm{Prop}|$
only when $d = 2$ and all carry transitions are accounted for in $\widetilde{T}$.
\item \emph{(General shadows satisfying \emph{(H1)} and \emph{(H2)}.)} If $d \ge 3$
and $\widetilde{T}$ satisfies \emph{(H1)} and \emph{(H2)}, then $a(L)$ does not admit
any Chebyshev parametrization.
\item \emph{($k$-summand addition.)} For $k$-summand base-$N$ addition with
$F = \{k-1\}$ and $N \ge 2$: the avoidance count is geometric for $k = 2$
(Part~(a), $d = 1$); it admits a Chebyshev parametrization for $k = 3$
(Part~(a), $d = 2$); and no Chebyshev parametrization exists for $k \ge 4$
(Part~(b), $d \ge 3$).
\end{enumerate}
\end{theorem}

\begin{proof}
\noindent\textit{Part~(a): $d \le 2 \Rightarrow$ exact closed form.}

\medskip
\noindent\textit{Case $d = 1$.}
$\widetilde{T} = (\tau)$ is a $1\times 1$ matrix.  The representation formula
$a(L) = \mathbf{1}^\top \widetilde{T}^L e_0$ gives $a(L) = \tau^L$ directly, since
$\widetilde{T}^L = (\tau^L)$.

\textit{Why no exact Chebyshev formula applies.}  Suppose $\tau^L = (\sqrt{\delta})^L U_L(x)$
for all $L$, with $\delta > 0$ and $x$ finite.  For $L=1$: $\tau = \sqrt{\delta}\cdot 2x$,
so $2x\sqrt{\delta} = \tau$.  For $L=2$: $\tau^2 = \delta\,U_2(x) = \delta(4x^2-1)$.
But $4x^2\delta = (2x\sqrt{\delta})^2 = \tau^2$, so $\tau^2 = \tau^2 - \delta$, giving
$\delta = 0$, a contradiction.  Hence the Chebyshev formula with $\delta > 0$ does not apply
for $d=1$; the geometric sequence $\tau^L$ is covered by part~(a) as a separate exact case.

\medskip
\noindent\textit{Case $d = 2$.}
Write $x := \tau/(2\sqrt{\delta})$ where $\tau = \operatorname{tr}(\widetilde{T})$ and
$\delta = \det(\widetilde{T})$.  The eigenvalues are
$\lambda_{1,2} = \sqrt{\delta}\bigl(x \pm \sqrt{x^2 - 1}\bigr)$.

\textit{Subcase $x > 1$ (distinct real roots).} By Theorem~\ref{thm:universality},
$a(L) = (\lambda_1^{L+1} - \lambda_2^{L+1})/(\lambda_1 - \lambda_2)$.
Set $\beta := \operatorname{arccosh}(x) > 0$, so $\lambda_{1,2} = \sqrt{\delta}\,e^{\pm\beta}$.
Then
\[
  \frac{\lambda_1^{L+1} - \lambda_2^{L+1}}{\lambda_1 - \lambda_2}
  = (\sqrt{\delta})^L\, \frac{\sinh((L+1)\beta)}{\sinh\beta}
  = (\sqrt{\delta})^L\, U_L(\cosh\beta)
  = (\sqrt{\delta})^L\, U_L(x),
\]
using the standard identity $U_L(\cosh\beta) = \sinh((L+1)\beta)/\sinh\beta$.

\textit{Subcase $x = 1$ (repeated root $\lambda = \sqrt{\delta}$).}
$a(L) = (L+1)(\sqrt{\delta})^L$. Since $U_L(1) = L+1$, we have $(\sqrt{\delta})^L U_L(1) = a(L)$.

\textit{Subcase $x < 1$ (complex roots, $x = \cos\theta$, $\theta \in (0,\pi)$).}
Write $\lambda_{1,2} = \sqrt{\delta}\,e^{\pm i\theta}$; the same computation with
$\beta = i\theta$ gives
$a(L) = (\sqrt{\delta})^L\, \sin((L+1)\theta)/\sin\theta = (\sqrt{\delta})^L\, U_L(x)$.

\medskip
\noindent\textit{Part~(b): $d \ge 3$ with (H1) and (H2) $\Rightarrow$ no Chebyshev.}

Define the generating function $A(z) := \sum_{L=0}^\infty a(L)\, z^L$.
Since $a(L) = \mathbf{1}^\top \widetilde{T}^L e_0$,
\[
A(z) = \mathbf{1}^\top(I - z\widetilde{T})^{-1} e_0 = \frac{P(z)}{Q(z)},
\]
where $Q(z) = \det(I - z\widetilde{T})$ has degree $d$ and $P(z) = \mathbf{1}^\top
\operatorname{adj}(I - z\widetilde{T})\, e_0$ has degree $\le d-1$.

By (H1), $Q(z)$ has $d$ distinct roots at $z_j = 1/\widetilde{\lambda}_j$.
By (H2), the partial-fraction residue of $A(z)$ at each $z_j$ is non-zero.
Hence $\gcd(P,Q) = 1$ and the denominator of $A(z)$ in reduced form has degree exactly $d$.

If $a(L) = (\sqrt{\fd})^L U_L(x)$ held for some $x \ge 1$ and $\fd > 0$, then with
$N' := 2x\sqrt{\fd}$ the recurrence $a(L) = N'a(L-1) - \fd\,a(L-2)$ gives
$A(z) = 1/(1 - N'z + \fd\,z^2)$.  The reduced denominator would have
degree exactly $2$, contradicting degree $d \ge 3$.

\medskip
\noindent\textit{Part~(c): $k$-summand addition.}
By Lemma~\ref{lem:simple-evals}, for $k$-summand base-$N$ addition with $F = \{k-1\}$,
$d = k - 1$ and (H1) holds. By Lemma~\ref{lem:nonzero-res}, (H2) holds for $k \ge 3$.
For $k = 2$, $d = 1$ and part~(a) gives the geometric formula $a(L) = N^L$.
For $k = 3$, $d = 2$ and part~(a) gives the Chebyshev parametrization. For
$k \ge 4$, $d = k - 1 \ge 3$ and part~(b) rules out any Chebyshev form.
\end{proof}

\begin{table}[ht]
\centering
\caption{Chebyshev threshold data for multi-summand addition.
Characteristic polynomials are for the restricted \emph{count} matrix
$\widetilde{T}_{\mathrm{count}} = N^k\widetilde{T}$;
for $k=4$, $N=2$, the numerical coefficients reflect eigenvalues of order
$N^{k-j}$.}
\label{tab:chebyshev-data}
\small
\begin{tabular}{@{}cccl@{}}
\toprule
$k$ & $d$ & Characteristic polynomial of $\widetilde{T}_{\mathrm{count}}$ & Chebyshev? \\
\midrule
$2$ & $1$ & $\lambda - N$ & Geometric ($a(L) = N^L$; no Chebyshev, see Thm.~\ref{thm:cheb-threshold}(a)) \\
$3$ & $2$ & $\lambda^2 - N\lambda + \fd$ & Yes \\
$4$ ($F = \{3\}$) & $3$ & $\lambda^3 - 25\lambda^2 + 165\lambda - 280$ & No (irred.\ over $\mathbb{Q}$) \\
$4$ ($F = \{2,3\}$) & $2$ & $\lambda^2 - 15\lambda + 40$ & Yes \\
\bottomrule
\end{tabular}
\end{table}

For $k = 4$, $N = 2$, $F = \{3\}$: Python verification confirms irreducibility of
$\lambda^3 - 25\lambda^2 + 165\lambda - 280$ over $\mathbb{Q}$ (no rational roots among
$\pm$divisors of $280$).

\begin{remark}[Fibonacci--Poisson coincidence and dispersion regimes]
\label{rem:fib-poisson}
\label{sec:poisson}
The coupling parameter $x = 3/2$ arises in two roles:
as the Chebyshev evaluation point for base-$3$ doubling, where
$a(L) = U_L(3/2) = F(2L+2)$ (Theorem~\ref{thm:fibonacci}; the identity follows
from the Binet formula and $\cosh^{-1}(3/2) = 2\ln\phi$, see~\cite{Koshy2001});
and as the Poisson transition point for symmetric carry chains ($g = \kl$), where
$D_\infty = 1$ at $\mu = 1/3$, i.e., $g = \kl = t = N/3$.
Both roles reduce to the same parameter condition; the minimal instance is
$N = 3$, $g = \kl = t = 1$.

More generally, the asymptotic dispersion index
$D_\infty := \lim_{L\to\infty}\operatorname{Var}(\nu)/\mathbb{E}[\nu]
= \pi_0(1+\mu)/(1-\mu)$,
where $\pi_0 = \kl/(g+\kl)$ and $\mu = t/N$
(\cite[Theorem~6.3]{Moj2026}),
equals~$1$ if and only if $2\kl t = g(g+\kl)$, dividing the parameter space
into overdispersed ($D_\infty > 1$), Poisson ($D_\infty = 1$), and
underdispersed ($D_\infty < 1$) regimes.
The symmetric case $g = \kl$ recovers the Fibonacci point above.
\end{remark}

\section{Stochastic Classification}
\label{sec:classification}

\subsection{Binary state spaces}

\begin{definition}[Shadow equivalence]
\label{def:equiv}
Two systems with binary carry state spaces ($k = 2$) and transfer matrices
$T^{(1)}$, $T^{(2)}$ are \emph{shadow-equivalent}, $\fS_1 \sim \fS_2$, if
$T^{(2)} = M T^{(1)} M^{-1}$ for some $M \in \mathrm{GL}(2,\mathbb{R})$.
\end{definition}

\begin{theorem}[Stochastic shadow classification]
\label{thm:stochastic-shadow}
Let $\fS_1$ and $\fS_2$ have binary Markovian carry chains ($k = 2$). The following are
equivalent:
\begin{enumerate}[label=(\roman*)]
\item $\fS_1 \sim \fS_2$ (shadow-equivalent).
\item $\fS_1$ and $\fS_2$ share the same pair $(N, \fd)$.
\end{enumerate}
\end{theorem}

\begin{proof}
\textit{(i)$\Rightarrow$(ii).} Similar matrices have equal traces and
determinants: $N_1 = \operatorname{tr}(T_1) = \operatorname{tr}(T_2) = N_2$
and $\fd_1 = \det(T_1) = \det(T_2) = \fd_2$.

\textit{(ii)$\Rightarrow$(i).} We construct $M \in \mathrm{GL}(2,\mathbb{R})$ explicitly.
The characteristic polynomial $\lambda^2 - N\lambda + \fd$ is the same for $T_1$ and
$T_2$ (Lemma~\ref{lem:transfer-inv}).

\textit{Case $N^2 > 4\fd$ (distinct real eigenvalues).} Let $\lambda_\pm =
(N \pm \sqrt{N^2 - 4\fd})/2$ be the common eigenvalues. Both $T_1$ and $T_2$ are
diagonalizable over $\mathbb{R}$: if $v_\pm^{(i)}$ are eigenvectors of $T_i$ for
$\lambda_\pm$, set $M_i := [v_+^{(i)} \mid v_-^{(i)}] \in \mathrm{GL}(2,\mathbb{R})$,
so $T_i = M_i \operatorname{diag}(\lambda_+, \lambda_-) M_i^{-1}$. Then
$T_2 = M_2 M_1^{-1} T_1 (M_2 M_1^{-1})^{-1}$, giving
$M := M_2 M_1^{-1} \in \mathrm{GL}(2,\mathbb{R})$ with $T_2 = M T_1 M^{-1}$.

\textit{Case $N^2 = 4\fd$ (repeated eigenvalue $\lambda = N/2$).} For a
$2 \times 2$ real matrix with a single eigenvalue $\lambda$ and minimal polynomial
$(\mu - \lambda)^2$ (i.e.\ not a scalar multiple of the identity), there is a unique
Jordan form $\begin{pmatrix}\lambda & 1 \\ 0 & \lambda\end{pmatrix}$. Any two such
matrices are similar over $\mathbb{R}$ via the change-of-basis relating their Jordan
bases.

\textit{Case $N^2 < 4\fd$ (complex eigenvalues).} The real canonical form for a
$2 \times 2$ real matrix with eigenvalues $\alpha \pm i\beta$ ($\beta \ne 0$) is
$\begin{pmatrix}\alpha & -\beta \\ \beta & \alpha\end{pmatrix}$; again any two real
matrices with the same complex eigenvalue pair are similar over $\mathbb{R}$.

In all three cases $M \in \mathrm{GL}(2,\mathbb{R})$ exists, so $\fS_1 \sim \fS_2$.
\end{proof}

\begin{corollary}[The moduli space]
\label{cor:moduli}
The pair $(N, \fd)$ is a complete invariant for the
$\mathrm{GL}(2,\mathbb{R})$-similarity class of $T$, and hence for all avoidance-related
quantities. The set of achievable pairs is
\[
  \Omega := \{(N, \fd) \in \mathbb{Z}_{>0} \times \mathbb{Z}_{\ge 0} :
             \exists\, g, t \in \mathbb{Z}_{\ge 0},\; gt = \fd,\; g + t \le N\}.
\]
The condition $N^2 \ge 4\fd$ is necessary for $(N,\fd) \in \Omega$ but not sufficient.
The precise realizability condition on $\fd \ge 1$ is $N \ge \sigma(\fd)$, where
$\sigma(\fd) := \min_{d \mid \fd,\, 1 \le d \le \sqrt{\fd}} (d + \fd/d)$.
\end{corollary}

\begin{proof}
Necessity of $N^2 \ge 4\fd$ follows from AM--GM applied to any realizing factorization
$gt = \fd$ with $g, t \ge 0$:
$N^2 - 4\fd = (g-t)^2 + \kl^2 + 2\kl(g+t) \ge 0$.
Sufficiency fails: the pair $(N, \fd) = (6, 7)$ satisfies $N^2 = 36 \ge 28 = 4\fd$, but
$\fd = 7$ is prime, so the only factorizations require $g + t = 8 > 6 = N$.
\end{proof}

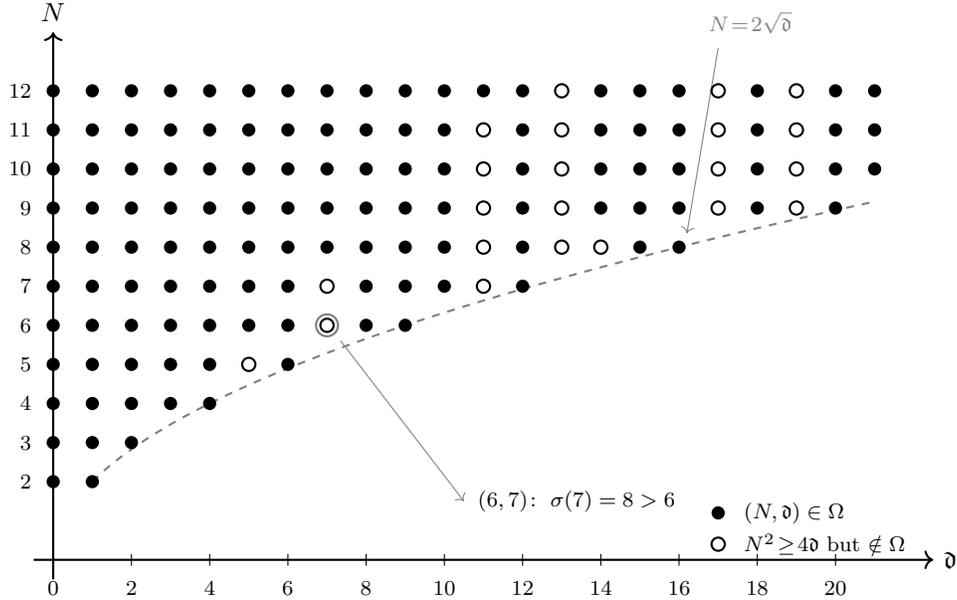
\begin{figure}[ht]
\centering
\begin{tikzpicture}[
    scale=0.52,
    ach/.style={fill=black, circle, inner sep=0pt, minimum size=5pt},
    nonach/.style={draw=black, circle, inner sep=0pt, minimum size=5pt,
                   fill=white, thick}
  ]
  \draw[->, thick] (-0.5,0) -- (22.5,0)
    node[right, font=\small] {$\fd$};
  \draw[->, thick] (0,-0.5) -- (0,13.5)
    node[above, font=\small] {$N$};
  \foreach \d in {0,2,...,20} {
    \draw (\d,-0.15) -- (\d,0.15);
    \node[below, font=\scriptsize] at (\d,-0.3) {$\d$};
  }
  \foreach \n in {2,3,...,12} {
    \draw (-0.15,\n) -- (0.15,\n);
    \node[left, font=\scriptsize] at (-0.3,\n) {$\n$};
  }
  \draw[dashed, gray, thick]
    plot[domain=1:21, samples=80, smooth] (\x, {2*sqrt(\x)});
  \node[gray, font=\scriptsize, anchor=south west] at (16.5,13.2)
    {$N\!=\!2\sqrt{\fd}$};
  \draw[gray, thin, ->] (17,13.1) -- ({16.2},{2*sqrt(16.2)+0.3});
  \foreach \d in {0,1} { \node[ach] at (\d,2) {}; }
  \foreach \d in {0,1,2} { \node[ach] at (\d,3) {}; }
  \foreach \d in {0,1,2,3,4} { \node[ach] at (\d,4) {}; }
  \foreach \d in {0,1,2,3,4,6} { \node[ach] at (\d,5) {}; }
  \node[nonach] at (5,5) {};
  \foreach \d in {0,1,2,3,4,5,6,8,9} { \node[ach] at (\d,6) {}; }
  \node[nonach] at (7,6) {};
  \foreach \d in {0,1,2,3,4,5,6,8,9,10,12} { \node[ach] at (\d,7) {}; }
  \node[nonach] at (7,7) {};
  \node[nonach] at (11,7) {};
  \foreach \d in {0,1,2,3,4,5,6,7,8,9,10,12,15,16}
    { \node[ach] at (\d,8) {}; }
  \node[nonach] at (11,8) {};
  \node[nonach] at (13,8) {};
  \node[nonach] at (14,8) {};
  \foreach \d in {0,1,2,3,4,5,6,7,8,9,10,12,14,15,16,18,20}
    { \node[ach] at (\d,9) {}; }
  \node[nonach] at (11,9) {};
  \node[nonach] at (13,9) {};
  \node[nonach] at (17,9) {};
  \node[nonach] at (19,9) {};
  \foreach \d in {0,1,2,3,4,5,6,7,8,9,10,12,14,15,16,18,20,21}
    { \node[ach] at (\d,10) {}; }
  \node[nonach] at (11,10) {};
  \node[nonach] at (13,10) {};
  \node[nonach] at (17,10) {};
  \node[nonach] at (19,10) {};
  \foreach \d in {0,1,2,3,4,5,6,7,8,9,10,12,14,15,16,18,20,21}
    { \node[ach] at (\d,11) {}; }
  \node[nonach] at (11,11) {};
  \node[nonach] at (13,11) {};
  \node[nonach] at (17,11) {};
  \node[nonach] at (19,11) {};
  \foreach \d in {0,1,2,3,4,5,6,7,8,9,10,11,12,14,15,16,18,20,21}
    { \node[ach] at (\d,12) {}; }
  \node[nonach] at (13,12) {};
  \node[nonach] at (17,12) {};
  \node[nonach] at (19,12) {};
  \draw[thick, gray] (7,6) circle (8pt);
  \draw[gray, thin, ->] (7.35,5.6) -- (10.5,1.5);
  \node[font=\scriptsize, anchor=west] at (10.6,1.5)
    {$(6,7)\!:\ \sigma(7)=8>6$};
  \node[ach] at (17,1.2) {};
  \node[font=\scriptsize, anchor=west] at (17.4,1.2)
    {$(N,\fd)\in\Omega$};
  \node[nonach] at (17,0.4) {};
  \node[font=\scriptsize, anchor=west] at (17.4,0.4)
    {$N^2\!\ge\!4\fd$ but $\notin\Omega$};
\end{tikzpicture}
\caption{The moduli space~$\Omega$ of shadow-equivalence classes
for binary carry chains ($k = 2$), shown for $N \le 12$ and
$\fd \le 21$.  Filled circles are achievable pairs
$(N, \fd) \in \Omega$; open circles satisfy the necessary
condition $N^2 \ge 4\fd$ but admit no integer factorization
$gt = \fd$ with $g + t \le N$.  The dashed curve is the
AM--GM bound $N = 2\sqrt{\fd}$.}
\label{fig:moduli}
\end{figure}

\subsection{General state spaces}

\begin{definition}[Class $\cM(k,N)$]
\label{def:mclass}
Let $\cM(k,N)$ denote the class of all Markovian carry chains of order~$1$ with
state space $\{0,\ldots,k-1\}$ and digit alphabet of size~$N$,
with $\mu$ any probability measure on the digit alphabet.
The transfer matrix $T \in \mathbb{R}^{k\times k}$ has entries
$T[c',c] = \mu(\{d : \phi(c,d) = c'\})$.
\end{definition}

\begin{example}[Non-trivial equivalence classes in $\cM(2,3)$]
\label{ex:M23}
For the GEN/PROP/KILL model with $k = 2$, $N = 3$, and digit measure $\mu$, set
$g := \mu(\mathrm{GEN})$, $t := \mu(\mathrm{PROP})$, $\ell := \mu(\mathrm{KILL})$ with
$g + t + \ell = 1$.  Consider four systems:
\begin{align*}
\text{System A (uniform):} &\quad g = \ell = t = \tfrac{1}{3},
  & \chi_A = \lambda^2 - \tfrac{4}{3}\lambda + \tfrac{1}{3}. \\[4pt]
\text{System B:} &\quad g = \tfrac{1}{2},\; t = \ell = \tfrac{1}{4},
  & \chi_B = \lambda^2 - \tfrac{5}{4}\lambda + \tfrac{1}{4}. \\[4pt]
\text{System C:} &\quad g = \ell = \tfrac{1}{4},\; t = \tfrac{1}{2},
  & \chi_C = \lambda^2 - \tfrac{3}{2}\lambda + \tfrac{1}{2}. \\[4pt]
\text{System D:} &\quad g = \tfrac{1}{6},\; t = \tfrac{1}{2},\; \ell = \tfrac{1}{3},
  & \chi_D = \lambda^2 - \tfrac{3}{2}\lambda + \tfrac{1}{2}.
\end{align*}
Systems A, B, C have pairwise distinct characteristic polynomials, hence are pairwise
non-equivalent. Systems C and D have identical characteristic polynomials despite
different digit distributions, so they are shadow-equivalent.
\end{example}

\begin{theorem}[General stochastic classification]
\label{thm:stoch-general}
Let $\fS_1, \fS_2 \in \cM(k,N)$ and assume both transfer matrices $T_1, T_2$ have $k$
pairwise distinct eigenvalues (simple spectrum). The following are equivalent:
\begin{enumerate}[label=(\roman*)]
\item $T_1$ and $T_2$ are $\mathrm{GL}(k,\mathbb{R})$-similar.
\item $T_1$ and $T_2$ have the same characteristic polynomial.
\item $T_1$ and $T_2$ have the same eigenvalue multiset.
\end{enumerate}
The moduli space within $\cM(k,N)$ with simple spectrum is
$(k-1)$-dimensional, parametrized by the $k-1$ non-leading coefficients of the
characteristic polynomial (the stochasticity constraint $\chi_T(1) = 0$ removes one
degree of freedom).
\end{theorem}

\begin{proof}
(i)$\Rightarrow$(ii): Similar matrices have equal characteristic polynomials.
(ii)$\Leftrightarrow$(iii): Standard.
(ii)$\Rightarrow$(i): Both $T_1$ and $T_2$, having simple spectrum, share the same
real canonical form~$J$ (each pair of complex conjugate eigenvalues contributes a
$2\times 2$ rotation-scaling block, each real eigenvalue a $1\times 1$ block).
If $T_i = P_i J P_i^{-1}$, then $M := P_2 P_1^{-1}$ gives $T_2 = M T_1 M^{-1}$.
\end{proof}

\begin{remark}[Uniform $k$-summand addition is a single point]
\label{rem:holte-degenerate}
When $\mu$ is the uniform measure on $k$-summand base-$N$ addition, the Holte matrix
depends only on $(k,N)$ (Theorem~\ref{thm:holte-spectrum}).
All systems in $\cM(k,N)$ with uniform digit measure are \emph{identical} as
Markov chains: the uniform subclass collapses to a single point of the moduli
space, with characteristic polynomial $\prod_{j=0}^{k-1}(\lambda - N^{-j})$.
The classification theorem is non-trivial precisely because $\cM(k,N)$ contains
systems with non-uniform digit measures, such as Systems C and D of
Example~\ref{ex:M23}.
\end{remark}

\section{Open Questions and Further Results}
\label{sec:open}

\begin{question}
\label{q:orthogonal-d3}
The Chebyshev threshold is $d \le 2$. For $d \ge 3$, is there a family of orthogonal
polynomials generalizing $U_L$ that parametrizes the avoidance counts? By
Proposition~\ref{prop:stirling-lagrange}, the characteristic
polynomial $\lambda^3 - 25\lambda^2 + 165\lambda - 280$ for $k = 4$, $N = 2$,
$F = \{3\}$ is irreducible over $\mathbb{Q}$.  The recurrence
$a(L) = 25\,a(L-1) - 165\,a(L-2) + 280\,a(L-3)$
with initial values $a(0) = 1$, $a(1) = 16$, $a(2) = 255$ yields the sequence
\[
1,\; 16,\; 255,\; 4015,\; 62780,\; 978425,\; 15226125,\; 236791400,\; \ldots
\]
(verified by transfer-matrix computation).
This sequence does not appear in OEIS as of April~2026.
\end{question}

\begin{remark}[Multiplicative shadow-freeness is impossible---a contrast]
\label{prop:mult-impossible}
While the classification above quantifies additive shadows, the multiplicative
case admits no shadow-free encoding at all:
for $N \ge 2$ and $L \ge 2$, no bijection
$h\colon \mathbb{Z}/N^L\mathbb{Z} \to (\mathbb{Z}/N\mathbb{Z})^L$ makes multiplication
shadow-free, i.e., there is no operation
$g\colon \mathbb{Z}/N\mathbb{Z} \times \mathbb{Z}/N\mathbb{Z} \to \mathbb{Z}/N\mathbb{Z}$
with $h(ab)_i = g(h(a)_i,\, h(b)_i)$ for all $a, b$ and all~$i$.

\begin{proof}
Suppose such $h$ and $g$ exist, so that $h$ is a monoid isomorphism
$(\mathbb{Z}/N^L\mathbb{Z}, \times) \xrightarrow{\sim}
((\mathbb{Z}/N\mathbb{Z})^L, g^{\mathrm{cw}})$.

\textit{Step~1: identity and absorbing elements.}
From $1 \cdot a = a \cdot 1 = a$: $h(1) = (e, \ldots, e)$ for a unique two-sided
identity~$e$.  From $0 \cdot a = 0$: $h(0) = (z, \ldots, z)$ for a unique absorbing
element~$z$.  Injectivity forces $e \ne z$.

\textit{Step~2: nilpotent count.}
A monoid isomorphism bijects nilpotent elements.  Write $N = \prod_{i=1}^r p_i^{a_i}$.
By CRT, $|\mathrm{Nil}(\mathbb{Z}/N^L\mathbb{Z})| = \prod_{i=1}^r p_i^{La_i - 1}$.
In $((\mathbb{Z}/N\mathbb{Z})^L, g^{\mathrm{cw}})$:
$|\mathrm{Nil}| = |\mathrm{Nil}_g|^L$ for some integer $|\mathrm{Nil}_g| \ge 1$.

\textit{Step~3: $p$-adic valuation obstruction.}
Equating: $|\mathrm{Nil}_g|^L = \prod_{i} p_i^{La_i - 1}$.  For each prime~$p_i$,
the left side has $v_{p_i} \equiv 0 \pmod{L}$, while the right side has
$v_{p_i} = La_i - 1 \equiv -1 \pmod{L}$.  For $L \ge 2$: $-1 \not\equiv 0 \pmod{L}$,
a contradiction.
\end{proof}
\end{remark}

\begin{remark}[Shadow uncertainty]
\label{rem:shadow-uncertainty}
For $N \ge 2$ and $L \ge 2$, every positional encoding~$h$ satisfies both
non-trivial additive shadows (since $\mathbb{Z}/N^L\mathbb{Z} \not\cong
(\mathbb{Z}/N\mathbb{Z})^L$ as groups)
and non-trivial multiplicative shadows
(Remark~\ref{prop:mult-impossible}):
neither addition nor multiplication can be made shadow-free.
The minimum of the total shadow cost over all encodings, and the
encodings achieving it, require techniques beyond
the spectral methods of this paper.
\end{remark}

\section*{Acknowledgments}

Iterative manuscript editing was
carried out with the assistance of Claude (Anthropic).
All mathematical content---definitions, theorem statements, proof
strategies, and final arguments---is the author's own;
the author takes full responsibility for the correctness
of all results.

\section*{Statements and Declarations}

\medskip
\noindent\textit{Funding.}
No funding was received for conducting this study.

\medskip
\noindent\textit{Competing Interests.}
The author has no competing interests to declare that are relevant to the
content of this article.

\medskip
\noindent\textit{Data Availability.}
The computational verification code for the transfer matrices, spectral
data, and eigenvector systems is available from the author upon reasonable
request.

\end{document}